\documentclass[11pt,reqno,twoside]{amsart}
\usepackage{graphicx}
\usepackage{amssymb}
\usepackage{amsthm}
\usepackage{epstopdf}
\usepackage[asymmetric,top=2.5cm,bottom=2.8cm,left=2.38cm,right=2.38cm]{geometry}
\usepackage{hyperref}

\usepackage{booktabs} 
\usepackage{array} 
\usepackage{paralist} 
\usepackage{verbatim} 
\usepackage{subfig} 
\usepackage{tabularx}
\usepackage{amsmath,amsfonts,amsthm,mathrsfs,amssymb,cite}
\usepackage[usenames]{color}
\usepackage{bm}
\usepackage{color}
\usepackage{extarrows}
\usepackage{xcolor}
\hypersetup{
	colorlinks,
	linkcolor={blue!80!black},
	urlcolor={blue!80!black},
	citecolor={blue!80!black}
}

\newtheorem{thm}{Theorem}[section]

\newtheorem{lem}{Lemma}[section]
\newtheorem{prop}{Proposition}[section]
\theoremstyle{definition}
\newtheorem{defn}{Definition}[section]
\theoremstyle{remark}

\newtheorem{rem}{Remark}[section]
\numberwithin{equation}{section}

\newcommand{\bu}{\mathbf{u}}

\newcommand{\bvarphi}{\bm{\varphi}}

\newcommand{\bnu}{\bm{\nu}}
\newcommand{\bGa}{\mathbf{\Gamma}}
\newcommand{\bx}{\mathbf{x}}
\newcommand{\by}{\mathbf{y}}
\newcommand{\rmi}{\mathrm{i}}
\newcommand{\bS}{\mathbf{S}}
\newcommand{\bK}{\mathbf{K}}
\newcommand{\bI}{\mathbf{I}}

\newcommand{\td}{\tilde}

\newcommand{\Acal}{\mathcal{A}}

\newcommand{\Lcal}{\mathcal{L}}

\newcommand{\Hcal}{\mathcal{H}}

\title[Quasi-Minnaert Resonances in 2D Elastic Wave Scattering]{Quasi-Minnaert Resonances in 2D Elastic Wave Scattering with Applications}

\author{Huaian Diao}
\address{School of Mathematics, Jilin University and Key Laboratory of Symbolic Computation and Knowledge Engineering of Ministry of Education, Changchun, Jilin, China.}
\email{diao@jlu.edu.cn, hadiao@gmail.com}

\author{Kaixin Lu}
\address{School of Mathematics, Jilin University, Changchun 130012, China.}
\email{lukaixin0215@163.com; lukx23@mails.jlu.edu.cn}

\author{Ruixiang Tang}
\address{School of Mathematics, Jilin University, Changchun 130012, China.}
\email{tangrx97@gmail.com; tangrx23@mails.jlu.edu.cn}

\author{Weisheng Zhou}
\address{School of Mathematics, Jilin University, Changchun 130012, China.}
\email{wszhou1211@163.com}

\begin{document}
	
	\maketitle
\begin{abstract}
		
In our earlier work \cite{DTL}, we introduced a novel quasi-Minnaert resonance for three-dimensional elastic wave scattering in the sub-wavelength regime. Therein, we provided a rigorous analysis of the boundary localization and surface resonance phenomena for both the total and scattered waves, achieved through carefully selected incident waves and tailored physical parameters of the elastic medium. In the present study, we focus on quasi-Minnaert resonances in the context of two-dimensional elastic wave scattering. Unlike the 3D case \cite{DTL}, the 2D setting introduces fundamental theoretical challenges stemming from (i) the intrinsic coupling between shear and compressional waves, and (ii) the complex spectral properties of the associated layer potential operators. By combining layer potential techniques with refined asymptotic analysis and strategically designed incident waves, we rigorously establish quasi-Minnaert resonances in both the internal and scattered fields. In addition, the associated stress concentration effects are quantitatively characterized.  Notably, the boundary-localized nature of the scattered field reveals potential applications in near-cloaking (via wave manipulation around boundaries). Our results contribute to a more comprehensive framework for studying resonance behaviors in high-contrast elastic systems.


	\medskip
		


\end{abstract}

\section{Introduction}
Metamaterials are artificially engineered composite structures, whose functionalities primarily arise from meticulously designed micro-geometries, rather than from the intrinsic material properties of their constituent components. Mathematical analysis plays a fundamental role in elucidating the underlying physical mechanisms that govern the behavior of such systems.
Among the sub-wavelength physical mechanisms underlying metamaterials, Minnaert resonance emerges as a classical phenomenon with a wide range of important applications \cite{DAC, LBFWT, LSL}. This resonance has garnered increasing attention in recent years \cite{AFGLZ, HBD, HFL, HZ, HZ5}.

In acoustics, the low-frequency resonance induced by the high density contrast between a bubble and the surrounding liquid is known as the Minnaert frequency. A rigorous mathematical study of Minnaert resonance in acoustic media was first conducted in~\cite{AFGL}, where a quantitative relation between the Minnaert frequency and the high-contrast parameters was established. The resonance behavior of bubbles embedded in soft elastic media was further investigated in~\cite{HLIHJ}, while the dipolar resonance induced by a hard material in a soft elastic background (known as the HISE structure) was studied in~\cite{lizou}, where an explicit formula for the resonance frequency of the hard material was derived. Furthermore, it has been shown in~\cite{ACC, DGS} that bubbles can act as contrast agents to exploit Minnaert resonance for recovering material parameters such as mass density and bulk modulus. Advances on time-domain inverse problems utilizing Minnaert resonance have been discussed in~\cite{MMSi, MMSINI, SS23, SSW24}. More recently, the resonance behavior has been rigorously characterized in both the time-harmonic and time-domain regimes via the resolvent operator theory~\cite{LS241, LS24, MPS}.

The aforementioned studies primarily focus on the analysis of effective negative physical parameters induced by resonance mechanisms, the dependence of resonance frequencies on high-contrast material parameters, and the application of these phenomena to inverse problems. Physical experiments~\cite{LAI, LIUCHAN, LIUZHANG} have demonstrated that, in the sub-wavelength regime, high contrast between the material and the background medium can induce significant boundary localization and pronounced oscillations near the material interface. To characterize such behavior, the concept of \emph{quasi-Minnaert resonance} was introduced in~\cite{DTL}. This resonance is characterized by boundary localization and surface resonance, offering a theoretical framework for explaining the elastic wave phenomena observed in physical experiments. In contrast to classical Minnaert resonance, which occurs at discrete resonant frequencies~\cite{AFGL,lizou}, the quasi-Minnaert resonance manifests over a continuous frequency spectrum.

In this study, we explore elastic wave propagation in the sub-wavelength regime, where the elastic scatterer's size is significantly smaller than the wavelength of the incident wave. Through appropriate coordinate rescaling, we normalize the scatterer's size to a unit reference scale, yielding an asymptotically small angular frequency for the incident wave. This multiscale approach underpins our analytical framework for investigating elastic wave propagation in high-contrast materials. Within this framework, we rigorously establish quasi-Minnaert resonance in both the internal total field and the external scattered field in a two-dimensional setting. This resonance is characterized by boundary localization and surface resonance. However, the two-dimensional analysis is notably more intricate and technically demanding than its three-dimensional counterpart. In three-dimensional elasticity, certain shear waves can be decoupled from other shear and compressional waves \cite{DTL}, simplifying the analysis of quasi-Minnaert resonance. In contrast, two-dimensional elasticity involves intrinsic coupling of all shear and compressional waves, significantly complicating the theoretical analysis. Moreover, the spectral properties of the associated layer potential operators in two dimensions are more complex, necessitating refined analytical techniques compared to the three-dimensional case. Although both analyses rely on asymptotic expansions in the frequency parameter $\omega$, the two-dimensional case exhibits a markedly more intricate structure. The asymptotic forms of the internal total field and external scattered field feature more complex leading-order coefficients than their three-dimensional counterparts. To address these challenges, we introduce novel asymptotic techniques tailored to the intricate coupling structure of two-dimensional elasticity, enhancing the robustness of our analytical approach.

In Theorem \ref{thm:isl}, we demonstrate that high-contrast materials subjected to appropriately designed incident waves exhibit boundary localization in the sub-wavelength regime. In Theorem \ref{thm:sre}, incident waves are constructed that simultaneously induce boundary localization and surface resonance, leading to the occurrence of quasi-Minnaert resonance. It is also proven that surface resonance can occur independently of wave field boundary localization by suitably choosing an incident wave tailored to the high-contrast material parameters. Finally, based on the analysis of quasi-Minnaert resonance, stress concentration phenomena in the interior total field and the exterior scattered field near the boundary of the hard material are rigorously characterized, as established in Theorem \ref{thm:eu}. The occurrence of quasi-Minnaert resonance arises from the subtle coupling between high-contrast material parameters and the incident field. Moreover, quasi-Minnaert resonance and surface resonance can induce the design of high-contrast material parameters, as detailed in our analysis. Additionally, it is noteworthy that the quasi-Minnaert resonance corresponds to a continuous frequency spectrum, in contrast to the discrete Minnaert resonances described in~\cite{AFGL,lizou}.

The quasi-Minnaert resonance framework developed in this work advances the theoretical understanding of elastic wave behavior in the sub-wavelength regime and leads to potential applications. It rigorously extends the quasi-Minnaert theory to two-dimensional elasticity, enriching the existing theoretical framework. The observed surface resonance and stress concentration phenomena suggest potential mechanisms for nearly cloaking, where the scattered field is localized around the boundary of the material, rendering the material nearly undetectable.  While the primary focus lies in theoretical analysis, the identified physical phenomena offer a solid foundation for future exploration of application-oriented directions.


The paper is organized as follows. In Section \ref{sec:setup}, we introduce the mathematical framework and foundational concepts of layer potential theory. In Section \ref{sec:alg}, we rigorously analyze the boundary localization behavior of both the internal total field and the external scattered field on the inner and outer boundaries of the material. In Section \ref{sec:experiments}, we explore the surface resonance phenomena associated with the internal total field and external scattered field. Furthermore, we provide a rigorous demonstration of the stress concentration effect observed in both the internal total field and the external scattered field.

\section{Mathematical setup}
\label{sec:setup}

In this section, we present the mathematical formulation for our subsequent study. We investigate the incorporation of a hard material within a soft elastic homogeneous background. The material, denoted by \( D \), is modeled as a bounded Lipschitz domain in \( \mathbb{R}^2 \), and its complement \( \mathbb{R}^2 \setminus \overline{D} \) is assumed to be connected. Consider the domain \( D \) characterized by the material parameters \( (\tilde{\lambda}, \tilde{\mu}, \tilde{\rho}) \), where \( \tilde{\rho} \in \mathbb{R}_{+} \) denotes the mass density, \( \tilde{\lambda} \) and \( \tilde{\mu} \) are the real  Lam\'e parameters  satisfying the strong convexity conditions:
\begin{align}\label{eq:con}
	\mathrm{i)}.~~\td{\mu}>0;\qquad\qquad \mathrm{ii)}.~~\td{\lambda}+\td{\mu}>0.
\end{align}
The soft elastic background is mathematically represented as the domain $\mathbb{R}^2\backslash\overline{D}$, which is filled with a homogeneous elastic medium. In this background, the corresponding Lam\'e parameters are $(\lambda, \mu)$, which satisfy the condition given in \eqref{eq:con}. The constant $\rho \in \mathbb{R}_{+}$ denotes the mass density of the background medium in $\mathbb{R}^2\backslash\overline{D}$.
Let $\bu^i$ be a time-harmonic incident elastic wave, which satisfies the elastic equation in  $\mathbb{R}^2$
\begin{equation}\label{eq:inci}
	\mathcal{L}_{ {\lambda}, {\mu}}\bu^i(\bx) + \omega^2  {\rho} \bu^i(\bx) =0,
\end{equation}
where $\omega>0$ is the angular  frequency and the Lam\'e operator $ \mathcal{L}_{\lambda, \mu}$ associated with the parameters $(\lambda,\mu)$ is defined by 
\begin{equation}\label{op:lame}
	\Lcal_{\lambda,\mu}\bu^i(\bx):=\mu \triangle\bu^i(\bx) + (\lambda+ \mu)\nabla\nabla\cdot\bu^i(\bx).
\end{equation}
The interaction between the incident wave \( \bu^i \) and the elastic material \( D \) generates a scattered elastic field \( \bu^s \), thereby forming the total  field \(\bu = \bu^i + \bu^s.\)
Then the total displacement field $\bu$ described above is controlled by the following system
\begin{equation}\label{eq:xtm}
	\left\{
	\begin{array}{ll}
		\mathcal{L}_{\td{\lambda}, \td{\mu}}\bu(\bx) + \omega^2 \td{\rho} \bu(\bx) =0,  &  \bx\in D,  \medskip \\
		\mathcal{L}_{\lambda, \mu}\bu(\bx) + \omega^2\rho\bu(\bx) =0,   &   \bx\in \mathbb{R}^2\backslash \overline{D},  \medskip \\
		\bu(\bx)|_- = \bu(\bx)|_+,     & \bx\in\partial D,  \medskip \\
		\partial_{\td{\bnu}}\bu(\bx)|_- = \partial_{{\bnu}}\bu(\bx)|_+, & \bx\in\partial D,  \medskip \\
		\bu^s:=\bu-\bu^i  \qquad \mbox{satisfies the radiation condition},
	\end{array}
	\right.
\end{equation} 
where the subscript $\pm$ indicate the limits from outside and inside of $D$, respectively. The angular frequency $\omega \in \mathbb{R}_{+}$. The co-normal derivative $\partial_{\bnu}$ associated with the parameters $(\lambda, \mu)$ in the system \eqref{eq:xtm} is defined as 
\begin{equation}\label{eq:trac}
	\partial_{\bnu}\bu=\lambda(\nabla\cdot \bu)\bnu + 2\mu(\nabla^s\bu) \bnu.
\end{equation}
Here, $\bnu$ is the outward unit normal to $\partial D$ and the operator $\nabla^s$ is the symmetric gradient
\begin{displaymath}
	\nabla^s\mathbf{u}:=\frac{1}{2}\left(\nabla\mathbf{u}+\nabla\mathbf{u}^{\top} \right),
\end{displaymath}
where $\nabla\bu$ denotes the matrix $(\partial_j u_i)_{i,j=1}^2$ and the superscript $\top$ signifies the matrix transpose.
The operators $ \mathcal{L}_{\td{\lambda}, \td{\mu}}$ and $\partial_{\td{\bnu}}$ are defined in \eqref{op:lame} and \eqref{eq:trac}, respectively, with the parameters $(\lambda, \mu)$ replaced by $(\td{\lambda}, \td{\mu})$.
In \eqref{eq:xtm}, the radiation condition is characterized by the following equations \cite{ LHLJL}:
\begin{align}
	(\nabla\times\nabla\times \bu^s)(\bx)\times\frac{\bx}{|\bx|}-\mathrm{i} {k}_s\nabla\times \bu^s(\bx)=&\mathcal{O}\left(|\bx|^{-2}\right),\notag\\
	\frac{\bx}{|\bx|}\cdot[\nabla(\nabla\cdot \bu^s)](\bx)-\mathrm{i} {k}_p\nabla \bu^s(\bx)=&\mathcal{O}\left(|\bx|^{-2}\right),\notag
\end{align}
as $|\mathbf{x}|\rightarrow+\infty$, where $\rmi$ denotes the imaginary unit and
\begin{equation}\label{pa:ksp}
	{k}_s=\omega/{c}_s, \quad  \quad {k}_p=\omega/{c}_p,
\end{equation}
with
\({c}_s = \sqrt{{\mu}/{\rho}}\) and  \({c}_p=\sqrt{ ({\lambda} + 2 {\mu})/{\rho}}\)
representing the compressional and shear wave speeds in the medium. It is well-known that the elastic wave can be decomposed into the shear wave (s-wave) and the compressional wave (p-wave). In \eqref{pa:ksp}, the parameters $k_s$ and $k_p$ signify the wavenumbers of the s-wave and p-wave, respectively. 
In what follows, the parameters $\td{k}_s, \td{k}_p, \td{c}_s, \td{c}_p$ are defined in \eqref{pa:ksp} by replacing $(\lambda, \mu, \rho)$ with $(\td{\lambda}, \td{\mu}, \td{\rho})$.
In this paper, we mainly focus on the configuration of incorporating hard material $D$ within a soft elastic medium. Accordingly, the respective physical parameters in different domains exhibit a high contrast and follow the subsequent relationship
\begin{align}\label{eq:hicon}
	\delta  = \frac{\lambda}{\td{\lambda} }=\frac{\mu}{\td{\mu}}, \quad  \epsilon =\frac{\rho}{\td{\rho}} ,
\end{align}
where $\delta\ll1$ and $\epsilon\ll 1$. To better describe the contrast of wave speeds in different regions, we introduce another parameter $\tau$, namely 
\begin{align}\label{eq:detau}
	\tau = \frac{{c}_s}{\td{c}_s} = \frac{{c}_p}{\td{c}_p}= \sqrt{\delta/\epsilon}.
\end{align}
The last equality in \eqref{eq:detau} is derived from \eqref{eq:hicon} and the expressions of $c_s$ and $c_p$ in \eqref{pa:ksp}. Moreover, we assume that the contrast $\tau$ satisfies 
\begin{align}\label{eq:de}
	\tau= \sqrt{\delta/\epsilon} < 1.
\end{align}
This assumption in \eqref{eq:de} is reasonable due to the fact that the wave speed within the hard material is higher than that within the soft background.
In this paper, we consider sub-wavelength regime, i.e.,
\begin{align}\label{ass:sub}
	\omega \cdot {\rm diam} (D) \ll 1.
\end{align}
The assumption \eqref{ass:sub} indicates that the size of the material $D$ is smaller than the operating wavelength of the incident wave. Assume that the physical configuration $(\lambda, \mu, \rho)$ describing the homogeneous elastic medium satisfies
\begin{align}\label{eq:lambda=}
	\lambda=\mathcal{O}(1),\quad\mu=\mathcal{O}(1),\quad\rho=\mathcal{O}(1).
\end{align}    
By coordinate transformation, we may assume that the size of the domain $D$ is of order $1$. Consequently, the angular frequency $\omega=o(1)$, which indicates that the compressional and shear wave numbers fulfill ${k}_s =o(1)$ and ${k}_p = o(1)$. Additionally, based on the relationship \eqref{eq:de}, we have $\td{k}_s = o(1), \td{k}_p = o(1)$ and 
\begin{align}\label{eq:tdksp}
	\td{k}_s = \tau k_s, \quad \td{k}_p =  \tau k_p.
\end{align}

In this study, we mainly apply the potential theory to investigate the system \eqref{eq:xtm}. Next, we present the potential theory of the Lam\'e system.
In the two-dimensional case, the Kupradze matrix $\bGa^{\omega}=(\Gamma^{\omega}_{i,j})_{i,j=1}^2$ of the fundamental solution to the operator $\Lcal_{\lambda,\mu} + \omega^2\rho$  is given by \cite{ABG}
\begin{displaymath}
	(\Gamma^{\omega}_{i,j})_{i,j=1}^2(\bx)=-\frac{\mathrm{i}}{4\mu}\delta_{ij} H_0\left(k_s|\bx|\right)+\frac{\mathrm{i}}{4\omega^2\rho}\partial_i\partial_j\left(H_0\left(k_p|\bx|\right)-H_0\left(k_s|\bx|\right)\right),
\end{displaymath}
where $H_{0}\left(\cdot\right)$ is the Hankel function of the first kind of order 0, and $k_s$ and $k_p$ are defined in \eqref{pa:ksp}, and $\delta_{ij}$ is Kronecker delta function. For $\omega=0$, we denote  $\bGa^{0}$ by  $\bGa$ for simplicity, and $\bGa^{0}$ has the following expression
\[\Gamma^{0}_{i,j} (\bx) = \frac{\gamma_1}{2\pi}\delta_{ij}\ln|\bx|-\frac{\gamma_2}{2\pi}\frac{x_ix_j}{|\bx|^2}. \]
The single-layer potential associated with the fundamental solution $\bGa^{\omega}$ is denoted by
\begin{displaymath}\label{eq:single}
	\bS_{\partial D}^{\omega}[\bvarphi](\bx)=\int_{\partial D} \bGa^{\omega}(\bx-\by)\bvarphi(\by)ds(\by), \quad \bx\in\mathbb{R}^2,
\end{displaymath}
for $\bvarphi\in L^2(\partial D)^2$. On the boundary $\partial D$, the co-normal derivative of the single-layer potential satisfies
\begin{align}\label{eq:jump}
	\partial_{\bnu}   \bS_{\partial D}^{\omega}[\bvarphi]|_{\pm}(\bx)=\left( \pm\frac{1}{2}\bI +  \bK_{\partial D}^{\omega, *} \right)[\bvarphi](\bx), \quad \bx\in\partial D,
\end{align}
where
\[\bK_{\partial D}^{\omega, *} [\bvarphi](\bx)=\mbox{p.v.} \int_{\partial D} \partial_{\bnu_{\bx}} \bGa^{\omega}(\bx-\by)\bvarphi(\by)ds(\by),\]
with $\mbox{p.v.}$ represents the Cauchy principal value. Here the operator $\bK_{\partial D}^{\omega, *}$ in \eqref{eq:jump} is called the Neumann-Poincar\'e (N-P) operator. 

With the help of the potential theory presented above, the solution of \eqref{eq:xtm} can be written in the following ansatz
\begin{align}\label{eq:sol}
	\bu=
	\left\{
	\begin{array}{ll}
		\td{\bS}_{\partial D}^{\omega}[\bvarphi_{1}](\bx), & \bx\in D,  \smallskip \\
		{\bS}_{\partial D}^{\omega}[\bvarphi_{2}](\bx) +\bu^i, &  \bx\in \mathbb{R}^2\backslash \overline{D},
	\end{array}
	\right.
\end{align}
where the density functions $\bvarphi_{1}, \bvarphi_{2} \in L^2(\partial D)^2$ are determined by the transmission condition across $D$. Here the operator $\td{\bS}_{\partial D}^{\omega}$ is the single-layer potential operator associated with the parameters $(\td{\lambda}, \td{\mu}, \td{\rho})$.
By applying the jump formula in \eqref{eq:jump} and imposing the transmission conditions on the boundary, the density functions $\bvarphi_{1}$ and $\bvarphi_{2}$ in \eqref{eq:sol} can be controlled by the following integral equations
\begin{align}\label{eq:or}
	\Acal(\omega,\delta) [\Phi](\bx)=F(\bx), \quad \bx\in\partial D,
\end{align}
where
\[
\Acal(\omega,\delta)=  \left(
\begin{array}{cc}
	\td{\bS}_{\partial D}^{\omega} &  -{\bS}_{\partial D}^{\omega} \medskip \\
	-\frac{\bI}{2} +  \td{\bK}_{\partial D}^{\omega, *} & -\frac{\bI}{2} -  {\bK}_{\partial D}^{\omega, *}\\
\end{array}
\right),
\;\;
\Phi= \left(
\begin{array}{c}
	\bvarphi_{1} \\
	\bvarphi_{2} \\
\end{array}
\right),
\;\; \mbox{and} \;\; 
F= \left(
\begin{array}{c}
	\bu^i \\
	\partial_{\bnu}  \bu^i \\
\end{array}
\right).
\]
For the further discussion, we introduce the spaces $\Hcal=L^2(\partial D)^2\times L^2(\partial D)^2$ and $\Hcal^1=H^1(\partial D)^2\times L^2(\partial D)^2$. The operator $\Acal(\omega,\delta)$ is defined from $\Hcal$ to $\Hcal^1$.

In this paper, we investigate the boundary localization and surface resonance for the elastic scattering problem described by \eqref{eq:xtm} in the sub-wavelength regime. To this end, we make suitable adaptations from the corresponding definitions in \cite{DTL}, adjusted to the two-dimensional case. 
\begin{defn}\label{def:surface localized}
	Recall that $D$ is a bounded Lipschitz domain with a connected complement. For any sufficiently small $\xi_{1}, \xi_{2} \in \mathbb{R}_{+}$, we define the neighborhoods in the interior and exterior of the boundary of $D$ as
	\begin{align}\label{eq:Mdef}
		\mathcal{N}_{D,\xi_{1}}^{-} &:= \{ \mathbf{x} \in D \,\big|\, \mathrm{dist}(\mathbf{x}, \partial D) < \xi_{1} \}, \\
		\mathcal{N}_{D,\xi_{2}}^{+} &:= \{ \mathbf{x} \in  D_R \setminus \overline{D} \,\big|\, \mathrm{dist}(\mathbf{x}, \partial D) < \xi_{2} \},\notag
	\end{align}
	where $\mathrm{dist}(\mathbf{x}, \partial D) := \inf\limits_{\mathbf{y} \in  \partial D} \|\mathbf{x} - \mathbf{y}\|$ denotes the Euclidean distance, and $D_R$ denotes a disk of radius $R$ centered at the origin in $\mathbb{R}^2$, such that $D \subseteq D_R$. Corresponding to the scattering problem \eqref{eq:xtm}, the internal total field $\mathbf{u}|_{D}$ is termed {\it interior boundary localized}, and the external scattered field $\mathbf{u}^s|_{\mathbb{R}^2 \setminus \overline{D}}$ is referred to as {\it exterior boundary localized}, provided there exist sufficiently small parameters $\xi_{1},\xi_{2}  \in \mathbb{R}_{+}$ corresponding to each case such that
	\begin{align}\label{eq:def2.1}
		\frac{\|\mathbf{u}\|_{L^{2}(D \setminus \mathcal{N}_{D,\xi_{1}}^{-})^{2}}^{2}}{\|\mathbf{u}\|_{L^{2}(D)^2}^{2}} \ll 1,
		\quad
		\frac{\|\mathbf{u}^{s}\|_{L^{2}(\left(D_R \setminus \overline{D}\right) \setminus \mathcal{N}_{D,\xi_{2}}^{+})^2}^{2}}{\|\mathbf{u}^{s}\|_{L^{2}(D_R \setminus \overline{D})^2}^{2}} \ll 1.
	\end{align}
\end{defn}
\begin{rem}
	According to Definition \ref{def:surface localized}, let the sufficiently small parameter \(\varepsilon\) characterize the level of boundary localization. We quantitatively express \eqref{eq:def2.1} as
	\begin{align}\label{eq:rem2.1}
		\frac{\|\mathbf{u}\|_{L^{2}(D \setminus \mathcal{N}_{D,\xi_{1}}^{-})^2}^{2}}{\|\mathbf{u}\|_{L^{2}(D)^2}^{2}} \leq \varepsilon\ll1, \quad \frac{\|\mathbf{u}^{s}\|_{L^{2}\left((D_R \setminus \overline{D}) \setminus \mathcal{N}_{D,\xi_{2}}^{+}\right)^2}^{2}}{\|\mathbf{u}^{s}\|_{L^{2}(D_R \setminus \overline{D})^2}^{2}} \leq \varepsilon\ll1.
	\end{align}
	By combining \eqref{eq:Mdef}, we can rewrite the form of \eqref{eq:rem2.1} as follows:
	\begin{align}\label{eq:rewrem2.1}
		\frac{\|\mathbf{u}\|_{L^2(\mathcal{N}_{D,\xi_{1}}^{-})^2}^2}{\|\mathbf{u}\|_{L^2(D)^2}^2} = 1 - \mathcal{O}\left(\varepsilon\right), \quad \frac{\|\mathbf{u}^s\|_{L^2(\mathcal{N}_{D,\xi_{2}}^{+})^2}^2}{\|\mathbf{u}^s\|_{L^2(\mathbb{R}^2 \setminus D)^2}^2} = 1 - \mathcal{O}\left(\varepsilon\right).
	\end{align}
	As more intuitively indicated by \eqref{eq:rewrem2.1}, for a sufficiently small \(\varepsilon\), if the index \(n\) of the incident wave in \eqref{eq:ui} satisfies the conditions stated in Theorem \ref{thm:isl}, then the internal total field \(\bu|_D\) and the external scattered field \(\bu^s|_{D_R \setminus \overline{D}}\) exhibit boundary localization in \(\mathcal{N}_{D,\xi_{1}}^{-}\) and \(\mathcal{N}_{D,\xi_{2}}^{+}\), respectively.
	This observation aligns with the boundary localization introduced in \eqref{eq:def2.1} of  Definition \ref{def:surface localized}.
\end{rem}

\begin{defn}\label{def:surface resonant}
	Consider the internal total field $\mathbf{u}|_{D}$ and the external scattered field $\mathbf{u}^s|_{\mathbb{R}^2 \setminus \overline{{D}}}$ corresponding to the scattering problem \eqref{eq:xtm} associated with  incident wave $\mathbf{u}^i$, where $D$ represents material with bounded Lipschitz boundary. The neighborhoods $\mathcal{N}_{D,\xi_{1}}^{-} $ and $\mathcal{N}_{D,\xi_{2}}^{+}$ are defined in \eqref{eq:Mdef}. If the following conditions hold: 
	\begin{align}\label{eq:def2.2}
		\frac{\|\nabla \mathbf{u} \|_{L^2(\mathcal{N}_{D,\xi_{1}}^{-})^2}}{\|\mathbf{u}^i\|_{L^2(D)^2} } \gg 1, 
		\qquad \mbox{and} \qquad
		\frac{\|\nabla \mathbf{u}^s \|_{L^2(\mathcal{N}_{D,\xi_{2}}^{+})^2} }{\|\mathbf{u}^i\|_{L^2(D)^2} } \gg 1,
	\end{align}
	then we say that $\mathbf{u}|_{D}$ and $\mathbf{u}^s|_{\mathbb{R}^2 \setminus \overline{{D}}}$ are surface resonances.
\end{defn}
\begin{rem}
	Comparing  Definitions \ref{def:surface localized} and \ref{def:surface resonant}, we observe that in \eqref{eq:def2.2}, the $L^2$ norm of the incident wave $\mathbf{u}^i$   is explicitly divided, whereas in \eqref{eq:def2.1}, it does not explicitly appear. Nevertheless, the $L^2$ norm of $\mathbf{u}^i$ is implicitly present in \eqref{eq:def2.1}, as the ratio in the inequalities remains unchanged when both numerator and denominator are divided by $\|\mathbf{u}^i\|_{L^2(D)^2}$. This operation effectively introduces normalization by $\|\mathbf{u}^i\|_{L^2(D)^2}$ for the wave field within the hard material $D$. The analysis focuses on the sub-wavelength regime characterized by $\omega \cdot {\rm diam} (D) \ll 1$, where the small volume of $D$ leads to a correspondingly small value of $\|\mathbf{u}^i\|_{L^2(D)^2}$. Under these conditions, normalization by $\|\mathbf{u}^i\|_{L^2(D)^2}$ becomes essential for proper characterization of the wave phenomena.
\end{rem}

Building on  Definitions \ref{def:surface localized} and \ref{def:surface resonant}, we introduce the definition of quasi-Minnaert resonance.
\begin{defn}\label{def:quasi minnaert}
	Consider the elastic scattering problem described by \eqref{eq:xtm} under the excitation of an incident wave \( \mathbf{u}^i \) at frequency \( \omega \). The frequency \( \omega \) is termed the \textit{quasi-Minnaert resonance frequency} associated with \( \mathbf{u}^i \) if the resulting internal total field \( \mathbf{u}|_D \) and external scattered field \( \mathbf{u}^s|_{\mathbb{R}^2 \setminus \overline{D}} \) exhibit boundary localization and surface resonance, respectively. In this case, the domain \( D \) is referred to as the \textit{quasi-Minnaert resonator} corresponding to the incident field \( \mathbf{u}^i \).
\end{defn}
\begin{rem} 
	The quasi-Minnaert resonance arises from the combined effect of the incident wave \( \mathbf{u}^i \) and the high contrast in material parameters. A key feature of this resonance is that the associated frequencies form a continuous set in the sub-wavelength regime. This behavior stands in contrast to the classical Minnaert resonance, as studied in \cite{AFGL,lizou}, which is solely determined by the material contrast, independent of the incident wave \( \mathbf{u}^i \), and produces a discrete set of resonance frequencies. Moreover, the quasi-Minnaert resonance offers a rigorous theoretical framework for explaining experimentally observed physical phenomena such as boundary localization and high surface oscillations of the generated wave field, as reported in \cite{LAI, LIUCHAN, LIUZHANG}. These insights underscore the significance of incorporating both incident wave and material microstructure in the analysis of elastic scattering resonances.
\end{rem}

\section{Boundary localization  of the interior total  and exterior scattered wave fields}
\label{sec:alg}
In this section, we first present the results related to the single-layer potential and the N-P operator $\bK_{\partial D}^{\omega, *}$ corresponding to a disk in $\mathbb{R}^2$. Next, we derive that in the sub-wavelength regime, by specifically choosing the incident wave as shown in \eqref{eq:ui}, the internal total field $\bu$ and the external scattered field $\bu^{s}$ of \eqref{eq:xtm} exhibit boundary localization both inside and outside the material $D$, as demonstrated in  Theorem \ref{thm:isl}. It is emphasize that by Theorem \ref{thm:isl}, the scattered wave  $\bu^{s}$  exhibits boundary localization in the vicinity of \(\partial D\) when one an appropriately choses the  incident wave. Therefore, the material \(D\) achieves near-cloaking under exterior measurements.

In the following discussion, we introduce the requisite notation and pertinent formulas. Let $\mathbb{N}$ be defined as the set of positive integers, and $\mathbb{N}_{0}=\mathbb{N}\cup\{0\}$. Let $J_{n}(t)$ and $H_n(t)$, $n\in\mathbb{N}$, denote the Bessel and Hankel functions of the first kind of order $n$, respectively. For any fixed $n\in\mathbb{N}_{0}$, if $0<|t|\ll1$, the following expansions hold (cf.\cite{CK}):
\begin{align}
	J_{n}(t) &=\frac{t^{n} }{2^{n}n! } \left [ 1-\frac{t^{2} }{4(n+1)}  +\frac{t^{4} }{2^{5}(n+2)(n+1) } + \mathcal{O}(t^4)\right ], \quad n\ge 0,\label{eq:jn}\\
	H_{n}(t)&=-\frac{\mathrm{i}2^{n} (n-1)!}{\pi t^{n} } \left [ 1+\frac{t^{2} }{4(n-1)} +\frac{t^{4} }{2^{5}(n-1)(n-2)}+\mathcal{O}(t^4)  \right ], \quad n\ge 4,\label{eq:hn}
\end{align}
and possesses the following recurrence relations 
\begin{align}
	J_{n}^{\prime}(t)=J_{n-1}(t)-(n/t)J_{n}(t),\quad
	J_{n}^{\prime}(t)=-J_{n+1}(t)+(n/t)J_{n}(t),\label{eq:dtgx}
\end{align}
and the function $H_n(t) $ also satisfies the same derivative identities.

To facilitate the presentation of our ideas throughout the rest of the paper, we focus on the case where the hard material \(D\) embedded in the soft elastic medium is a unit disk.  Correspondingly, let $D_R \subset \mathbb{R}^2$ denote the disk centered at the origin with radius $R \in \mathbb{R}_+$.  Let $\bx=\left(x_{j}\right)_{j=1}^{2} \in \mathbb{R}^{2}$  be the Euclidean coordinates and $r=|\bx|$. Let $\theta_{\bx}$ be the angle between $\bx$ and the $x_{1}-$axis; without ambiguity we use $\theta$ instead of $\theta_{\bx}$. We consider the disk $D_{R}$, where $\bm{\upsilon}=(\cos(\theta),\sin(\theta))^t$ represents the  outward unit normal to a boundary $\partial D_{R}$ and $\bm{t}=(-\sin(\theta),\cos(\theta))^{t}$ denotes the tangential direction along $\partial D_{R}$. Next, we derive the expressions for the single-layer potentials $\mathbf{S}_{\partial D_R}^\omega $ associated with the two densities $e^{\mathrm{i}n\theta}\bm{\upsilon}$ and $e^{\mathrm{i}n\theta}\bm{t}$.

Lemmas \ref{lem:sinlay} and \ref{lem:npz} state some spectral properties of the single-layer potential and the N-P operator associated with two densities \(e^{\mathrm{i}n\theta}\bm{\upsilon}\) and \(e^{\mathrm{i}n\theta}\bm{t}\).
A detailed proof for the case \(\rho = 1\) is provided in \cite[Theorem 1, Proposition 1, Lemma 8-10]{HHZJ}. A similar proof can be established for a general constant \(\rho\). Hence, the detailed proof of Lemmas \ref{lem:sinlay} and \ref{lem:npz} is omitted. Lemma \ref{lem:sinlay} characterizes the action of the single layer potential operator on the vector fields $e^{\mathrm{i}n\theta}\bm{\upsilon}$ and $e^{\mathrm{i}n\theta}\bm{t} $.

\begin{lem}\label{lem:sinlay}
	The single-layer potentials $\mathbf{S}_{\partial D_{R}}^{\omega}[e^{\mathrm{i}n\theta}\bm{\upsilon}]$ and $\mathbf{S}_{\partial D_{R}}^{\omega}[e^{\mathrm{i}n\theta}\bm{t}]$ have the following expressions  for $|\bx|=R$:
	\begin{align}
		\mathbf{S}_{\partial D_{R}}^{\omega}\left[e^{\mathrm{i} n \theta}\bm{\upsilon }\right](\bx)&=\alpha_{1 n} e^{\mathrm{i} n \theta} \bm{\upsilon}+\alpha_ {2 n} e^{\mathrm{i} n \theta} \bm{t},\label{equ:svt} \\ 
		\mathbf{S}_{\partial D_{R}}^{\omega}\left[e^{\mathrm{i} n \theta} \bm{t}\right](\bx)&=\alpha_{3 n} e^{\mathrm{i} n \theta} \bm{\upsilon}+\alpha_{4 n} e^{\mathrm{i} n \theta} \mathbf{t},\notag
	\end{align}
	where
	\begin{align}
		\alpha_{1n} &=-\frac{\mathrm{i}\pi}{2\omega^{2}\rho R}\left(n^{2}J_{n}(k_{s}R)H_{n}(k_{s}R)+k_{p}^{2}R^{2}J_{n}^{\prime}(k_{p}R)H_{n}^{\prime}(k_{p}R)\right), \notag\\
		\alpha_{2n} &=\frac{n\pi}{2\omega^{2}\rho}\left(k_{s}J_{n}(k_{s}R)H_{n}^{\prime}(k_{s}R)+k_{p}J_{n}^{\prime}(k_{p}R)H_{n}(k_{p}R)\right), \notag\\
		\alpha_{3n} &=-\frac{n\pi}{2\omega^{2}\rho}\left(k_{s}J_{n}^{\prime}(k_{s}R)H_{n}(k_{s}R)+k_{p}J_{n}(k_{p}R)H_{n}^{\prime}(k_{p}R)\right), \notag\\
		\alpha_{4n} &=-\frac{\mathrm{i}\pi}{2\omega^{2}\rho R}\left(k_{s}^{2}R^{2}J_{n}^{\prime}(k_{s}R)H_{n}^{\prime}(k_{s}R)+n^{2}J_{n}(k_{p}R)H_{n}(k_{p}R)\right). \notag
	\end{align}
	
	The single-layer potentials $\mathbf{S}_{\partial D_{R}}^{\omega}[e^{\mathrm{i}n\theta}\bm{\upsilon}]$ and $\mathbf{S}_{\partial D_{R}}^{\omega}[e^{\mathrm{i}n\theta}\bm{t}]$ have the following expressions for $\bx\in\mathbb{R}^{2}\backslash\overline{D}_{R}$:
	\begin{align}\label{lem:qninutw}
		\mathbf{S}_{\partial D_{R}}^{\omega}[e^{\mathrm{i}n\theta}\bm{\upsilon }](\mathbf{x})&= \frac{-\mathrm{i}\pi}{4\omega^{2}\rho R}\left(nk_{s}RJ_{n}(k_{s}R)\bm{\Psi}_{n}^{s,o}(k_{s}|\mathbf{x}|)+k_{p}^{2}R^{2}J_{n}^{\prime}\big(k_{p}R\big)\bm{\Psi}_n^{p,o}(k_{p}|\mathbf{x}|)\right), \\
		\mathbf{S}_{\partial D_{R}}^{\omega}[e^{\mathrm{i}n\theta}\mathbf{t}](\mathbf{x})&= \frac{-\pi}{4\omega^{2}\rho R}\left(k_{s}^{2}R^{2}J_{n}^{\prime}(k_{s}R)\bm{\Psi}_{n}^{s,o}(k_{s}|\mathbf{x}|)+nk_{p}RJ_{n}(k_{p}R)\bm{\Psi}_n^{p,o}(k_{p}|\mathbf{x}|)\right), \notag
	\end{align}
	where
	\begin{align}
		\bm{\Psi}_{n}^{s,o}(k_{s}|\mathbf{x}|)&= \frac{2nH_{n}(k_{s}|\mathbf{x}|)}{k_{s}|\mathbf{x}|}e^{\mathrm{i}n\theta}\bm{\upsilon }+2\mathrm{i}H_{n}^{\prime}(k_{s}|\mathbf{x}|)e^{\mathrm{i}n\theta}\bm{t},\label{eq:qp0} \\
		\bm{\Psi}_n^{p,o}(k_p|\mathbf{x}|)&= 2H_{n}^{\prime}(k_{p}|\mathbf{x}|)e^{\mathrm{i}n\theta}\bm{\upsilon}+\frac{2\mathrm{i}nH_{n}(k_{p}|\mathbf{x}|)}{k_{p}|\mathbf{x}|}e^{\mathrm{i}n\theta}\bm{t}. \notag
	\end{align}
	Moreover, the function $\bm{\Psi}_{n}^{s,o}(k_{s}|\mathbf{x}|)$  pertains to the s-wave, and  $\bm{\Psi}_n^{p,o}\left(k_p|\mathbf{x}|\right)$ pertains to the p-wave. Both of them are radiating solutions to the equation $(\mathcal{L}_{\lambda,\mu}+\omega^{2}{\rho} )\mathbf{u}=0$ in $\bx\in\mathbb{R}^{2}\backslash\overline{D}_{R}$.
	
	Besides, the single-layer potentials $\mathbf{S}_{\partial D_{R}}^{\omega}[e^{\mathrm{i}n\theta}\bm{\upsilon}]$ and $\mathbf{S}_{\partial D_{R}}^{\omega}[e^{\mathrm{i}n\theta}\bm{t}]$ have the following expressions for $\bx\in B_R$:
	\begin{align}\label{lem:qninut}
		\mathbf{S}_{\partial D_{R}}^{\omega}[e^{\mathrm{i}n\theta}\bm{\upsilon }](\bx)&=\frac{-\mathrm{i}\pi}{4\omega^{2}\rho R}\left(nk_{s}RH_{n}(k_{s}R)\bm{\Psi}_{n}^{s,i}(k_{s}|\bx|)+k_{p}^{2}R^{2}H_{n}^{\prime}\big(k_{p}R\big)\bm{\Psi}_{n}^{p,i}(k_{p}|\mathbf{x}|)\right),\\
		\mathbf{S}_{\partial D_{R}}^{\omega}[e^{\mathrm{i}n\theta}\bm{t}](\bx)&=\frac{-\pi}{4\omega^{2}\rho R}\left(k_{s}^{2}R^{2}H_{n}^{\prime}(k_{s}R)\bm{\Psi}_{n}^{s,i}(k_{s}|\bm{x}|)+nk_{p}RH_{n}(k_{p}R)\bm{\Psi}_{n}^{p,i}(k_{p}|\bm{x}|)\right),\notag
	\end{align}
	where
	\begin{align}
		\bm{\Psi}_{n}^{s,i}(k_{s}|\bm{x}|)&=\frac{2nJ_{n}(k_{s}|\bm{x}|)}{k_{s}|\bm{x}|}e^{\mathrm{i}n\theta}\bm{\upsilon}+2\mathrm{i}J_{n}^{\prime}(k_{s}|\bm{x}|)e^{\mathrm{i}n\theta}\bm{t},\label{lem:qni}\\
		\bm{\Psi}_{n}^{p,i}(k_{p}|\bm{x}|)&=2J_{n}^{\prime}(k_{p}|\bm{x}|)e^{\mathrm{i}n\theta}\bm{\upsilon}+\frac{2\mathrm{i}nJ_{n}(k_{p}|\bm{x}|)}{k_{p}|\bm{x}|}e^{\mathrm{i}n\theta}\bm{t}.\notag
	\end{align}
	Moreover, the function $\bm{\Psi}_{n}^{s,i}(k_{s}|\mathbf{x}|)$  pertains to the s-wave, and the function $\bm{\Psi}_{n}^{p,i}(k_p|\mathbf{x}|)$ pertains to the p-wave. Both of them are entire solutions to the equation $(\mathcal{L}_{\lambda,\mu}+\omega^{2}{\rho} )\mathbf{u}=0$ in $\bx\in{D}_{R}$.
\end{lem}
\begin{lem}\label{lem:npz}
	The N-P operator $\mathbf{K}_{\partial D_{R}}^{\omega,*}$ have the following expressions with two densities $e^{\mathrm{i}n\theta}\bm{\upsilon}$ and $e^{\mathrm{i}n\theta}\bm{t}$,
	\begin{align}
		\mathbf{K}_{\partial D_{R}}^{\omega,*}\left[e^{\mathrm{i}n\theta}\bm{\upsilon}\right]&=a_{1n}e^{\mathrm{i}n\theta}\bm{\upsilon}+a_{2n}e^{\mathrm{i}n\theta}\bm{t},\label{equ:kvt}\\
		\mathbf{K}_{\partial D_{R}}^{\omega,*}\left[e^{\mathrm{i}n\theta}\bm{t}\right]&=b_{1n}e^{\mathrm{i}n\theta}\bm{\upsilon}+b_{2n}e^{\mathrm{i}n\theta}\bm{t},\notag
	\end{align}
	where
	\begin{displaymath}
		a_{1n}=-\frac{1}{2}+g_{1,n}(R),\quad a_{2n}=g_{2,n}(R),\quad b_{1n}=g_{3,n}(R),\quad b_{2n}=-\frac{1}{2}+g_{4,n}(R).
	\end{displaymath}
	The coefficients $g_{i,n}(R)$ ($1 \leq i \leq 4$) can be derived from the tractions $\partial_{\nu}\mathbf{S}_{\partial D_{R}}^{\omega}[e^{\mathrm{i}n\theta}\bm{\upsilon}]\big|_{+}$ and $\partial_{\nu}\mathbf{S}_{\partial D_{R}}^{\omega}[e^{\mathrm{i}n\theta}\bm{t}]\big|_{+}$ evaluated on $\partial D_{R}$, as defined by the following expressions
	\begin{align}
		\partial_{\nu}\mathbf{S}_{\partial D_{R}}^{\omega}[e^{\mathrm{i}n\theta}\bm{\upsilon}]|_{+}&=g_{1,n}(|\bx|)e^{\mathrm{i}n\theta}\bm{\upsilon}+g_{2,n}(|\bx|)e^{\mathrm{i}n\theta}\bm{t},\notag\\
		\partial_{\nu}\mathbf{S}_{\partial D_{R}}^{\omega}[e^{\mathrm{i}n\theta}\bm{t}]|_{+}&=g_{3,n}(|\bx|)e^{in\theta}\bm{\upsilon}+g_{4,n}(|\bx|)e^{in\theta}\bm{t},\notag
	\end{align}
	where
	\begin{align}
		g_{1,n}(|\bx|)=&\frac{\mathrm{i}\pi}{2\omega^{2}\rho R^{2}}\Big(2\mu n^{2}J_{n}(k_{s}R)\big(H_{n}(k_{s}|\bx|)-k_{s}RH_{n}^{\prime}(k_{s}|\bm{x}|)\big)+ \notag\\
		&J_n^{\prime}(k_pR)k_pR(H_n(k_p|\bx|)(\omega^2\rho R^2-2\mu n^2)+2k_p\mu RH_n^{\prime}(k_p|\bx|))\Big), \notag\\
		g_{2,n}(|\bx|)=& - \frac{n\mu\pi}{2\omega^{2}\rho R^{2}}\Big(J_{n}(k_{s}R)H_{n}(k_{s}|\bx|)\Big(k_{s}^{2}R^{2}-2n^{2}\Big)+ \notag\\
		&2R\left(k_{s}J_{n}(k_{s}R)H_{n}^{\prime}(k_{s}|\bx|)+k_{p}J_{n}^{\prime}\left(k_{p}R\right)\left(H_{n}(k_{p}|\bx|)-k_{p}RH_{n}^{\prime}(k_{p}|\bx|)\right)\right)\Big), \notag\\
		g_{3,n}(|\bx|)=& \frac{n\pi}{2\omega^{2}\rho R^{2}}\big(J_{n}\big(k_{p}R\big)H_{n}\big(k_{p}\big|\mathbf{x}\big|\big)\big(\omega^{2}\rho R^{2}-2\mu n^{2}\big)+ \notag\\
		&2\mu R\big(k_pJ_n\big(k_pR\big)H_n'\big(k_p|\bx|\big)+k_sJ_n'(k_sR)\big(H_n(k_s|\bx|)-k_sRH_n'(k_s|\bx|)\big)\big)\big), \notag\\
		g_{4,n}(|\bx|)=& \frac{\mathrm{i}\mu\pi}{2\omega^{2}\rho R^{2}}\big(2n^{2}J_{n}\big(k_{p}R\big)\big(H_{n}\big(k_{p}|\bx|\big)-k_{p}RH_{n}^{\prime}\big(k_{p}|\bm{x}|\big)\big)+ \notag\\
		&J_n^{\prime}(k_sR)k_sR\left(k_s^2R^2H_n(k_s|\bx|)+2k_sRH_n^{\prime}(k_s|\bx|)-2n^2H_n(k_s|\bx|)\right).\notag
	\end{align}
\end{lem}

We present the selection of an incident wave that is pertinent to our study. Based on  Lemmas \ref{lem:sinlay} and \ref{lem:npz}, and through direct computations, we obtain that the function $\bm{\Psi}_{n}^{s,i}$  defined in \eqref{lem:qni} satisfies the elastic equation $\Lcal_{\lambda,\mu}\bm{\Psi}_{n}^{s,i}+\omega^{2}\rho\bm{\Psi}_{n}^{s,i}=0$. Therefore, we select $\bu^{i}$, which is an entire solution to \eqref{eq:inci}, as shown below
\begin{equation}\label{eq:ui}
	{\bu}^{i}=\kappa \bm{\Psi}_{n}^{s,i},
\end{equation}
where $\kappa \in \mathbb{C}$ is nonzero constant. 

In Theorem \ref{thm:isl}, under the assumptions \eqref{eq:hicon}-\eqref{eq:de} and \eqref{eq:lambda=}, we rigorously prove that both the interior total field \( \mathbf{u}|_D \) and the exterior scattered field \( \mathbf{u}^s|_{\mathbb{R}^2 \setminus \overline{D}} \) of \eqref{eq:xtm} are boundary localized in the sense of  Definition \ref{def:surface localized}, provided that the incident wave \( \mathbf{u}^i \) is chosen as in \eqref{eq:ui} with a sufficiently large \( n \). Recall the definitions of \(\mathcal{N}_{D,\xi_{1}}^{-} \) and \(\mathcal{N}_{D,\xi_{2}}^{+} \) provided in \eqref{eq:Mdef}, where in  Theorem \ref{thm:isl} we set \(\xi_1 = 1 - \gamma_1\) and \(\xi_2 = \gamma_2 - 1\), with \(\gamma_1 \in (0, 1)\) and \(\gamma_2 \in (1, R)\) being constants. To facilitate the proof of Theorem \ref{thm:isl}, we present the following lemma.
\begin{lem}\label{lem:thmxy}
	Consider the constants $\mathcal{R}_1$ and $\mathcal{R}_2$ as defined in \eqref{eq:R1} and \eqref{eq:R2}. We introduce the constants \(K_1\) and \(K_2\) as:
	\begin{align}\label{eq:lemk1k2}
		K_{1} = \frac{2(1-2\gamma_{1}^2+\gamma_{1}^4)}{\mathcal{R}_{1}}, \quad K_{2} =\frac{2\left(1-2\gamma_{2}^2+\gamma_{2}^4\right)}{\mathcal{R}_{2}},
	\end{align}
	where \(\gamma_1 \in (0, 1)\) and \(\gamma_2 \in (1, R)\) are constants. We formally define the integers \(n_2\) and \(n_4\) as follows:
	\begin{align}\label{eq:n24}
		n_{2}=\left\{
		\begin{array}
			{ll}n_{K_{1}}, &K_{1} > \frac{e^2(\ln\gamma_{1})^2}{4}, \\
			1, & K_{1} \leq \frac{e^2(\ln\gamma_{1})^2}{4},
		\end{array}\right.\quad
		n_{4}=\left\{
		\begin{array}
			{ll}n_{K_{2}}, &K_{2} > \frac{e^2(\ln\gamma_{2})^2}{4}, \\
			1, & K_{2} \leq \frac{e^2(\ln\gamma_{2})^2}{4},
		\end{array}\right.
	\end{align}
	where
	\begin{displaymath}
		n_{K_{1}}= \left\lfloor\frac{2}{\ln\gamma_{1}}W_{-1}\left(\frac{\ln\gamma_{1}}{2\sqrt{K_{1}}}\right)\right\rfloor+1,\quad n_{K_{2}} = \left\lfloor\frac{-2}{\ln\gamma_{2}}W_{-1}\left(\frac{-\ln\gamma_{2}}{2\sqrt{K_{2}}}\right)\right\rfloor+1,
	\end{displaymath}
	with \(W_{-1}\left(\cdot\right)\) denoting a branch of the Lambert \(W\) function.
	Consequently, if \( n \geq n_2 \), then \( K_1 n^2 \leq \gamma_1^{-n} \); similarly, if \( n \geq n_4 \), then \( K_2 n^2 \leq \gamma_2^{n} \).
\end{lem}
\begin{proof}
	The parameters \(\lambda\), \(\mu\), and \(\tau\) are defined in \eqref{eq:hicon}-\eqref{eq:de}. Under the conditions \eqref{eq:lambda=} and  $\tau < 1$, it follows that $\mathcal{R}_{1} = \mathcal{O}(1)$ and $\mathcal{R}_{2} = \mathcal{O}(1)$. 
	
	To facilitate the proof of Lemma \ref{lem:thmxy}, we introduce the Lambert $W$ function \cite{RGD}, a multivalued function in complex analysis defined as the inverse branches of the transcendental equation $f(W) = We^W$. For each integer $k$, the $k$-th branch $W_k(z)$ satisfies the defining relation:
	\[ z = W_k(z)e^{W_k(z)} \quad \text{for} \quad z \in \mathbb{C}. \]
	For real-valued arguments, the two most relevant branches are \( W_0 \) and \( W_{-1} \): the principal branch \( W_0(z) \in [-1, +\infty) \) is defined for real \( z \in [-1/e, +\infty) \), whereas \( W_{-1}(z) \in (-\infty, -1] \) is defined for real \( z \in [-1/e, 0) \). By employing the Lambert \( W \) function and performing detailed computations, we obtain the following results. When \( K_1 \leq \frac{e^2(\ln \gamma_1)^2}{4} \), the inequality \( K_1 n^2 \leq \gamma_1^{-n} \) holds for all natural numbers \( n \in \mathbb{N} \). In the complementary case where \( K_1 > \frac{e^2(\ln \gamma_1)^2}{4} \), there exists a critical index
	\[
	n_{K_1} = \left\lfloor \frac{2}{\ln \gamma_1} W_{-1} \left( \frac{\ln \gamma_1}{2\sqrt{K_1}} \right) \right\rfloor + 1,
	\]
	such that the inequality holds for all \( n \geq n_{K_1} \). Similarly, when \( K_2 \leq \frac{e^2(\ln \gamma_2)^2}{4} \), the inequality \( K_2 n^2 \leq \gamma_2^n \) holds for all \( n \geq 1 \). In the case where \( K_2 > \frac{e^2(\ln \gamma_2)^2}{4} \), the corresponding critical index is given by
	\[
	n_{K_2} = \left\lfloor \frac{-2}{\ln \gamma_2} W_{-1} \left( \frac{-\ln \gamma_2}{2\sqrt{K_2}} \right) \right\rfloor + 1.
	\]
	Based on the detailed computations above, it follows that if \( n \geq n_2 \), then \( K_1 n^2 \leq \gamma_1^{-n} \), and if \( n \geq n_4 \), then \( K_2 n^2 \leq \gamma_2^{n} \).
\end{proof}

We are now prepared to present the theorem concerning boundary localization.
\begin{thm}\label{thm:isl}
	Consider the elastic scattering problem described by \eqref{eq:xtm}, which involves an elastic material \(\left(D;\tilde{\lambda},\tilde{\mu},\tilde{\rho}\right)\) embedded in  homogeneous elastic medium \(\left(\mathbb{R}^2 \setminus \overline{D};\lambda,\mu,\rho\right)\), where \(D\) is a unit disk centered at the origin in \(\mathbb{R}^2\). Let \(D_R\) represent a disk of radius \(R\), centered at the origin in \(\mathbb{R}^2\), such that \(D \subseteq D_R\).  Let \(\bu\) denote the internal total wave field in \(D\), and  \(\bu^s\) represent the external scattered field in \(\mathbb{R}^2 \setminus \overline{D}\), both of which satisfy \eqref{eq:xtm} for an incident wave \(\bu^i\) defined in \eqref{eq:ui}. Under the assumptions \eqref{eq:hicon}-\eqref{eq:de} and \eqref{eq:lambda=}, for any \(\gamma_1 \in (0, 1)\), \(\gamma_2 \in (1, R)\), and a sufficiently small \(\varepsilon \in \mathbb{R}_+\), if the index \(n\) associated with \(\bu^i\) satisfies the condition 
	\begin{align}\label{eq:condition}
		n \geq \max(n_1, n_2, n_3, n_4),
	\end{align} where
	\begin{align}\label{eq:n13}
		n_1 = \left\lfloor\frac{\ln\varepsilon}{\ln\gamma_1}\right\rfloor + 1, \quad n_3 = \left\lfloor-\frac{\ln\varepsilon}{\ln\gamma_2}\right\rfloor + 1,
	\end{align}
	with $n_2$ and $n_4$ defined in \eqref{eq:n24}
	then we have
	$$
	\frac{\|\mathbf{u}\|_{L^2(D \setminus \mathcal{N}_{D,1-\gamma_{1}}^{-})^2}^2}{\|\mathbf{u}\|_{L^2(D)^2}^2} \leq \varepsilon + \mathcal{O}\left(\varepsilon\omega^2\right) \ll 1, \quad \frac{\|\mathbf{u}^s\|_{L^2(\left(D_R \setminus \overline{D}\right) \setminus \mathcal{N}_{D,\gamma_{2}-1}^{+})^2}^2}{\|\mathbf{u}^s\|_{L^2(\mathbb{R}^2 \setminus D)^2}^2} \leq \varepsilon + \mathcal{O}\left(\varepsilon\omega^2\right)  { \ll }  1.
	$$
	Here, the quantity \(\varepsilon\) characterizes the boundary localization level, while \(\lfloor \cdot \rfloor\) denotes the floor function. 
\end{thm}

\begin{proof}
	We consider the scattering problem described by \eqref{eq:xtm}, with the incident wave defined in \eqref{lem:qni}. According to the expressions for the function \(\bm{\Psi}_{n}^{s,i}\) in \eqref{lem:qni}, one has that on the boundary \(\partial D\):
	\begin{displaymath}
		{\bu}^{i}=f_{1,n}e^{\mathrm{i}n\theta}\bm{\upsilon }+f_{2,n}e^{\mathrm{i}n\theta}\bm{t },
	\end{displaymath}
	where
	\begin{displaymath}
		f_{1,n}=\frac{2n\kappa J_n(k_{s})}{k_{s}},\quad f_{2,n}=2\mathrm{i}\kappa J_{n}'(k_{s}).
	\end{displaymath}
	Moreover, by combining the definition of \(\partial_\nu\) provided in \eqref{eq:trac} with direct computations, we obtain the following on the boundary \(\partial D\):
	\begin{displaymath}
		\partial_\nu\bu^{i}=\tilde{f}_{1,n}e^{\mathrm{i}n\theta}\bm{\upsilon }+\tilde{f}_{2,n}e^{\mathrm{i}n\theta}\bm{t },
	\end{displaymath}
	with \begin{align}
		\tilde{f}_{1,n}&=\kappa \gamma_{1,n},\quad \gamma_{1n}=\frac{4n\mu}{k_s}\left(k_sJ_n'(k_s)-J_n(k_s)\right),\notag\\
		\tilde{f}_{2,n}&=\kappa \gamma_{2,n},\quad \gamma_{2n}=\frac{2\mathrm{i}\mu}{k_s}\left(\left(2n^2-k_s^2\right)J_n(k_s)-2k_sJ_n'(k_s)\right).\notag
	\end{align}
	Hence, the density functions \(\bm{\varphi}_{1}\) and \(\bm{\varphi}_{2}\) under the basis \(\left(e^{\mathrm{i}n\theta}\bm{\upsilon}, e^{\mathrm{i}n\theta}\bm{t}\right)\) can be expressed as
	\begin{align}\label{eq:phizh}
		\bm{\varphi}_1=\varphi_{1,1,n}e^{\mathrm{i}n\theta}\bm{\upsilon }+\varphi_{1,2,n}e^{\mathrm{i}n\theta}\bm{t },\quad\ \bm{\varphi}_2=\varphi_{2,1,n}e^{\mathrm{i}n\theta}\bm{\upsilon }+\varphi_{2,2,n}e^{\mathrm{i}n\theta}\bm{t }.
	\end{align}
	
	Using the jump formula \eqref{eq:jump}, along with \eqref{equ:svt} and \eqref{equ:kvt}, equation \eqref{eq:or} can be reformulated as:
	\begin{align}\label{eq:AXb}\bm{Ax}=\bm{b},\end{align}
	where
	\begin{align}\label{eq:axb}
		\bm{A}=\begin{bmatrix}
			\tilde{\alpha}_{1n} & \tilde{\alpha}_{3n} & -\alpha_{1n} & -\alpha_{3n} \\
			\tilde{\alpha}_{2n} & \tilde{\alpha}_{4n} & -\alpha_{2n} & -\alpha_{4n} \\
			\tilde{g}_{1n}-1 & \tilde{g}_{3n} & -g_{1n} & -g_{3n} \\
			\tilde{g}_{2n} & \tilde{g}_{4n}-1 & -g_{2n} & -g_{4n} \\
		\end{bmatrix},\quad
		\bm{x}=\begin{bmatrix}\varphi _{1,1,n} 
			\\\varphi _{1,2,n} 
			\\\varphi _{2,1,n} 
			\\\varphi _{2,2,n} 		
		\end{bmatrix},
		\quad
		\bm{b}=\begin{bmatrix}f _{1,n} 
			\\f _{2,n}  
			\\\tilde{f}_{1,n} 
			\\\tilde{f}_{2,n}  
		\end{bmatrix}.
	\end{align}
	Here, the parameters $\alpha_{in}$ and $g_{in}$ ($i = 1, 2, 3, 4$) are defined in \eqref{lem:sinlay} and \eqref{lem:npz}, while the parameters $\tilde{\alpha}_{in} $ and $\tilde{g}_{in}$  ($i = 1, 2, 3, 4$) are similarly defined in \eqref{lem:sinlay} and \eqref{lem:npz}, with $(\lambda, \mu)$ replaced by $(\tilde{\lambda}, \tilde{\mu})$.
	
	In the sub-wavelength regime, i.e., $\omega \ll 1$, by combining \eqref{eq:jn}, \eqref{eq:hn}, and \eqref{eq:axb}, the aforementioned equation \eqref{eq:AXb} are equivalently expressed as follows
	
	In the sub-wavelength regime, we perform an asymptotic expansion with respect to \(\omega\), combining the expansions of \(J_n\) and \(H_n\) defined in \eqref{eq:jn} and \eqref{eq:hn}. The aforementioned equation \eqref{eq:AXb} is equivalently expressed as follows:
	\begin{align}\label{eq:xtzh}				\left(\bm{A_0}+\omega^{2}\bm{A_{2}}+\omega^{4}\bm{A_{4}}+\mathcal{O}\left(\frac{\omega^{6}}{n^{6}}\right)\right)\bm{x}=\bm{b},
	\end{align}
	where
	\begin{align}
		\left.\bm{A_0}=\left[\begin{array}{cccc}
			\frac{-n  \tau ^2 (\lambda +3 \mu )}{4 \mu  \left(n^2-1\right) (\lambda +2 \mu )} & \frac{\mathrm{i}  \tau ^2 (\lambda +3 \mu )}{4 \mu  \left(n^2-1\right) (\lambda +2 \mu )} & \frac{n  (\lambda +3 \mu )}{4 \mu  \left(n^2-1\right) (\lambda +2 \mu )} & \frac{-\mathrm{i}  (\lambda +3 \mu )}{4 \mu  \left(n^2-1\right) (\lambda +2 \mu )} \\
			\frac{-\mathrm{i}  \tau ^2 (\lambda +3 \mu )}{4 \mu  \left(n^2-1\right) (\lambda +2 \mu )} & \frac{-n  \tau ^2 (\lambda +3 \mu )}{4 \mu  \left(n^2-1\right) (\lambda +2 \mu )} & \frac{\mathrm{i}  (\lambda +3 \mu )}{4 \mu  \left(n^2-1\right) (\lambda +2 \mu )} & \frac{n  (\lambda +3 \mu )}{4 \mu  \left(n^2-1\right) (\lambda +2 \mu )} \\
			-\frac{1}{2} & -\frac{\mathrm{i}}{2}  \left(1-\frac{\tau ^2 (\lambda +3 \mu )}{\lambda +2 \mu }\right) & -\frac{1}{2} & \frac{-\mathrm{i} \mu }{2 \lambda +4 \mu } \\
			\frac{-\mathrm{i} \mu  \tau ^2}{2 \lambda +4 \mu } & \frac{1}{2} \left(\tau ^2-2\right) & \frac{\mathrm{i} \mu }{2 \lambda +4 \mu } & -\frac{1}{2} \\
		\end{array}\right.\right].\notag
	\end{align}
	We noticed that 
	$det\left ( A_{0}  \right ) =\frac{ \left(\tau ^2+1\right) (\lambda +3 \mu )^2}{32 \mu ^2 \left(n^2-1\right) (\lambda +2 \mu )^2}\ne0, n\ge 2 $. Through direct calculations, one has 
	\begin{align}\notag
		\bm{A}_{0}^{-1} =\begin{bmatrix}
			b_{11} & b_{12} & \frac{-2}{\tau ^2+1} & \frac{\mathrm{i} \left(1-\tau ^2\right)}{\tau ^2+1} \\
			\frac{2 \mathrm{i} \mu  (\lambda +\mu  (n+2))}{R (\lambda +3 \mu )} & \frac{2 \mu  (\mu +n (\lambda +2 \mu ))}{- (\lambda +3 \mu )} & 0 & -1 \\
			b_{31} & b_{32} & \frac{2}{\tau ^2+1}-2 & \frac{-\mathrm{i} \tau ^2 \left(\tau ^2-1\right)}{\tau ^2+1} \\
			\frac{2 \mathrm{i} \mu  \left(-2 \lambda -4 \mu +\tau ^2 (\lambda +\mu  (n+2))\right)}{ (\lambda +3 \mu )} & \frac{2 \mu  \left(\mu  \tau ^2+n \left(\tau ^2-2\right) (\lambda +2 \mu )\right)}{-(\lambda +3 \mu )} & 0 & -\tau ^2 \\
		\end{bmatrix},
	\end{align}
	with
	\begin{align}
		b_{11}&=\frac{2 \mu  \left(\lambda -\left(\tau ^2 (\lambda +\mu  (n+2))\right)-2 \lambda  n-3 \mu  n\right)}{ \left(\tau ^2+1\right) (\lambda +3 \mu )},\notag\\
		b_{12}&=-\frac{2 \mathrm{i} \mu  \left(2 \lambda +3 \mu +\tau ^2 (\mu +n (\lambda +2 \mu ))-\lambda  n\right)}{ \left(\tau ^2+1\right) (\lambda +3 \mu )},\notag\\
		b_{31}&=\frac{2 \mu  \left(-\left(\tau ^4 (\lambda +\mu  (n+2))\right)+\tau ^2 (\lambda +\mu  n)+2 n (\lambda +2 \mu )\right)}{ \left(\tau ^2+1\right) (\lambda +3 \mu )},\notag\\
		b_{32}&=\frac{2 \mathrm{i} \mu  \left(2 \lambda +4 \mu -\tau ^4 (\mu +n (\lambda +2 \mu ))+\tau ^2 (\mu +\lambda  n)\right)}{ \left(\tau ^2+1\right) (\lambda +3 \mu )}.\notag
	\end{align}
	Thus, the equation \eqref{eq:xtzh} can be reformulated as 
	\begin{align}
		\mathbf{A_0}\left(\mathbf{I}+\omega^2\mathbf{A_0}^{-1}\mathbf{A_2}+\omega^4\mathbf{A_0}^{-1}\mathbf{A_4}+\mathcal{O}\left(\frac{\omega^6}{n^5}\right)\right)\bm{x}=\mathbf{b}.\notag
	\end{align}
	Here, since \(\omega \ll 1\), and the elements of the inverse matrix $\mathbf{A_0}^{-1}$ as well as those of matrix $\mathbf{A_2}$ are bounded, we can obtain that
	\[
	\left\|\omega^2\mathbf{A_0}^{-1}\mathbf{A_2}+\omega^4\mathbf{A_0}^{-1}\mathbf{A_4}+\mathcal{O}\left(\frac{\omega^6}{n^5}\right)\right\|_{\infty}<1.
	\]

	Consequently, by combining \eqref{eq:jn}, \eqref{eq:hn}, and \eqref{eq:dtgx}, we obtain the asymptotic expansion of the solution as:
	\begin{align}
		\bm{x}=&\left(\mathbf{I}+\omega^2\mathbf{A_0}^{-1}\mathbf{A_2}+\omega^4\mathbf{A_0}^{-1}\mathbf{A_4}+\mathcal{O}\left(\frac{\omega^6}{n^5}\right)\right)^{-1}\mathbf{A_{0}^{-1}b}\notag \\
		=& \mathbf{A_0^{-1}b}-\omega^2\left(\mathbf{A_0^{-1}A_2 A_0^{-1}b}\right)+\mathcal{O}\left(\frac{\omega^4}{n^2}\left(\mathbf{A_0^{-1}b}\right)\right) \notag\\
		=&\begin{pmatrix}t_{1} 
			\\t_{2} 
			\\t_{3} 
			\\t_{4} 	
		\end{pmatrix} -\omega^2\begin{pmatrix}t_{5}
			\\t_{6}
			\\t_{7}
			\\t_{8}		
		\end{pmatrix}+\mathcal{O}\left(\frac{\omega^{4}}{n^2}\begin{pmatrix}t_{1} 
			\\t_{2} 
			\\t_{3} 
			\\t_{4} 
		\end{pmatrix}\right),\notag
	\end{align}
	where
	\begin{align}
		t_{1}&=2 \kappa \mu  \left(\frac{ 2\left(\tau ^2-3\right) (\lambda +2 \mu )(n-1)\left(  k_s J_n'\left(k_s\right)+ n J_n\left(k_s\right)\right) }{\left(\tau ^2+1\right) (\lambda +3 \mu ) k_s} -\frac{ \left(\tau ^2-1\right)  k_s J_n\left(k_s\right)}{\left(\tau ^2+1\right)}\right),\notag\\
		t_{2}&=-2 \mathrm{i} \kappa\mu  \left(\frac{ 2 (n-1) (\lambda +2 \mu ) (k_s J_n'\left(k_s\right)+nJ_n\left(k_s\right)}{(\lambda +3 \mu ) k_s}- k_s J_n\left(k_s\right)\right),\notag\\
		t_{3}&=4 \kappa\mu  \frac{  (\lambda +2 \mu )  \left((n-1) \tau ^4+(1-3 n) \tau ^2-2\right) k_sJ_n'\left(k_s\right)}{\left(\tau ^2+1\right) (\lambda +3 \mu ) k_s}-2 \kappa\mu  \frac{ \tau ^2 \left(\tau ^2-1\right)  k_s J_n\left(k_s\right)}{\left(\tau ^2+1\right) }\notag\\
		&\quad +4 \kappa\mu  \frac{  n (\lambda +2 \mu ) \left((n-1) \tau ^4-(n-3) \tau ^2+2 n\right) J_n\left(k_s\right)}{\left(\tau ^2+1\right) (\lambda +3 \mu ) k_s},\notag\\
		t_{4}&=-4 \mathrm{i} \kappa \mu  \frac{  (\lambda +2 \mu ) (n-1) \tau ^2(k_sJ_n'\left(k_s\right)+nJ_n\left(k_s\right)) }{(\lambda +3 \mu ) k_s}+2 \mathrm{i} \kappa \mu  \tau ^2  k_s J_n\left(k_s\right)\notag\\
		&\quad-8 \mathrm{i} \kappa \mu  \frac{ n (\lambda +2 \mu ) (k_sJ_n'\left(k_s\right)-J_n\left(k_s\right))}{(\lambda +3 \mu ) k_s}.\notag
	\end{align}
	Accordingly, we derive the coefficients in \eqref{eq:phizh}, one can obtain 
	\begin{align}
		\varphi _{1,1,n}=&\frac{\kappa 2^{3-n} \left(\tau ^2-3\right) (\lambda +2 \mu ) \mu ^{\frac{3}{2}-\frac{n}{2}} \rho ^{\frac{n}{2}-\frac{1}{2}}}{\left(\tau ^2+1\right) (\lambda +3 \mu ) (n-2)!}\omega ^{n-1}+q_{1}\omega ^{n+1}+ \mathcal{O}(\omega^{n+3}),\label{eq:x12} \\
		\varphi _{1,2,n}=&-\frac{\mathrm{i} \kappa 2^{3-n} (\lambda +2 \mu ) \mu ^{\frac{3}{2}-\frac{n}{2}} \rho ^{\frac{n}{2}-\frac{1}{2}}}{(\lambda +3 \mu ) (n-2)!}\omega ^{n-1}+q_{2}\omega ^{n+1}+ \mathcal{O}(\omega^{n+3}),\notag
	\end{align}
	where
	\begin{align}\notag
		q_{1}=\mathcal{O}\left(\frac{\kappa 2^{1-n}  \mu ^{\frac{1}{2}-\frac{n}{2}} \rho ^{\frac{n}{2}+\frac{1}{2}}}{(n-1)!}\right),\quad q_{2}=\mathcal{O}\left(\frac{\mathrm{i} \kappa 2^{1-n}  \mu ^{\frac{1}{2}-\frac{n}{2}} \rho ^{\frac{n}{2}+\frac{1}{2}}}{(n-1)!}\right).
	\end{align}
	
	Based on \eqref{eq:sol}, by substituting the coefficients from \eqref{eq:x12} into formula \eqref{eq:phizh} and combining it with  \eqref{lem:qninut} in Lemma \ref{lem:sinlay}, we obtain
	\begin{align}
		\bu =&\tilde{S}_{\partial D }^{\omega }[\bm{\varphi} _{1} ] (\bx) 
		=\varphi _{1,1,n}\tilde{S}_{\partial D }^{\omega }[e^{\mathrm{i}n\theta }\bm{\upsilon}] (\bx)+\varphi _{1,2,n}\tilde{S}_{\partial D }^{\omega }[e^{\mathrm{i}n\theta }\bm{t}] (\bx) \label{eq:usw}\\
		=&\varphi _{1,1,n}\frac{-\mathrm{i}\pi}{4\omega^{2}\tilde{\rho} }\big(n\tilde{k}_{s}H_{n}(\tilde{k}_{s})\bm{\Psi}_{n}^{s,i}(\tilde{k}_{s}|\bx|)+\tilde{k}_{p}^{2}H_{n}^{\prime}\big(\tilde{k}_{p}\big)\bm{\Psi}_{n}^{p,i}(\tilde{k}_{p}|\mathbf{x}|)\big)\notag\\
		&+\varphi _{1,2,n}\frac{-\pi }{4\omega ^{2} \tilde{\rho}} \left ( n\tilde{k}_{p}H_{n}\left (\tilde{k}_{p} \right ) \bm{\Psi}_{n}^{p,i} \left (\tilde{k}_{p} \left | \bx \right |  \right ) +\tilde{k}_{s}^{2}  H_{n}'\left ( \tilde{k}_{s}  \right )  \bm{\Psi}_{n}^{s,i} \left (\tilde{k}_{s} \left | \bx \right |  \right ) \right ).\notag
	\end{align}
	Subsequently, by substituting  \eqref{lem:qni} from Lemma \ref{lem:sinlay} into \eqref{eq:usw}, applying the asymptotic expansions given in \eqref{eq:jn}, \eqref{eq:hn}, and \eqref{eq:dtgx}, and performing the necessary calculations, one obtains the following asymptotic expansion for $\bu$:
	\begin{align}\label{eq:uae}
		\bu
		=&\left(\left(\xi _1|\bx|^{n-1}+\xi_2|\bx|^{n+1}\right)e^{\mathrm{i}n\theta }\bm{\upsilon}+\left(\xi _3|\bx|^{n-1}+\xi_4|\bx|^{n+1}\right)e^{\mathrm{i}n\theta }\bm{t}\right)\omega ^{n-1}\left(1+\mathcal{O}\left(\omega^2\right)\right),
	\end{align}
	where
	\begin{align}
		\xi _1&=\frac{\kappa 2^{1-n} \tau ^2 \mu ^{\frac{1}{2}-\frac{n}{2}} \rho ^{\frac{n}{2}-\frac{1}{2}} \epsilon\left( (\lambda  (n+1)+\mu  (n+5))-\left(n-1\right) \tau ^2 (\lambda +\mu )\right)}{ \left(\tau ^2+1\right) (\lambda +3 \mu ) (n-1)!},\notag\\
		\xi _2&=\frac{\kappa 2^{1-n} \tau ^2 \mu ^{\frac{1}{2}-\frac{n}{2}} \rho ^{\frac{n}{2}-\frac{1}{2}} \epsilon  \left(\tau ^2-1\right)  (n (\lambda +\mu )-2 \mu )}{(n+1) \left(\tau ^2+1\right) (\lambda +3 \mu ) (n-2)!},\notag\\
		\xi _3&=\frac{\mathrm{i}\kappa 2^{1-n} \tau ^2 \mu ^{\frac{1}{2}-\frac{n}{2}} \rho ^{\frac{n}{2}-\frac{1}{2}}\epsilon \left( (\lambda  (n+1)+\mu  (n+5))-\left(n-1\right) \tau ^2 (\lambda +\mu )\right)}{ \left(\tau ^2+1\right) (\lambda +3 \mu ) (n-1)!},\notag\\
		\xi _4&=\frac{\mathrm{i} \kappa 2^{1-n} \tau ^2  \mu ^{\frac{1}{2}-\frac{n}{2}} \rho ^{\frac{n}{2}-\frac{1}{2}} \epsilon (n-1)\left(\tau ^2-1\right)(\lambda  (n+2)+\mu  (n+4))}{(n+1) \left(\tau ^2+1\right) (\lambda +3 \mu )  (n-1)!}.\notag
	\end{align}
	Furthermore, through rigorous computations, one can obtain
	\begin{align}
		\xi _1&+\xi _2+\xi _3+\xi _4=\frac{(1+\mathrm{i}) \kappa 2^{1-n} \tau ^2 \epsilon  \mu ^{\frac{1-n}{2}} \rho ^{\frac{n-1}{2}} \left(\mathrm{i} (n-1) \tau ^2+(2-\mathrm{i}) n+(2+\mathrm{i})\right)}{(n+1) \left(\tau ^2+1\right)  (n-1)!}.\notag
	\end{align}		
	Based on \eqref{eq:uae}, we can obtain 
	\begin{align}
		\|\bu\|_{L^2(D \setminus \mathcal{N}_{D,1-\gamma_{1}}^{-})^2}^2
		=&\left(\zeta _{1,n} \gamma_1 ^{2n}+\zeta _{2,n} \gamma_1 ^{2n+2}+\zeta _{3,n} \gamma_1 ^{2n+4}\right)\omega ^{2n-2}\left(1+\mathcal{O}(\omega^{2})\right)\notag\\
		=&\zeta _{1,n}\gamma_1 ^{2n}\left(1+\zeta _{4,n} \gamma_1 ^{2}+\zeta _{5,n} \gamma_1 ^{4}\right)\omega ^{2n-2}\left(1+\mathcal{O}(\omega^{2})\right),\label{eq:unorm}
	\end{align}
	where
	\begin{align}
		&\zeta _{1,n}=\frac{\pi  \kappa^2 2^{3-2 n} \tau ^4 \mu ^{1-n} \rho ^{n-1} \epsilon^{2}\left(\lambda +5 \mu -\left((n-1) \tau ^2 (\lambda +\mu )\right)+\lambda  n+\mu  n\right)^2}{n \left(\tau ^2+1\right)^2 (\lambda +3 \mu )^2  (n-1)!^2},\label{eq:zetau}\\
		&\zeta_{2,n}=\frac{\left(\lambda +5 \mu -(n-1) \tau ^2 (\lambda +\mu )+\lambda  n+\mu  n\right)}{(n+1) \left(\tau ^2+1\right)^2 (\lambda +3 \mu )^2(n-2)!(n-1)!}\notag\\
		&\quad\quad\times \pi  \kappa^2 4^{2-n} \tau ^4 \left(\tau ^2-1\right) \epsilon ^2 (\lambda +\mu ) \mu ^{1-n} \rho ^{n-1},\notag\\
		&\zeta_{3,n}=\frac{\pi  \kappa^2 4^{1-n} (n-1)^2 \tau ^4 \left(\tau ^2-1\right)^2 \epsilon ^2 \mu ^{1-n} \rho ^{n-1} }{(n+1)^2 (n+2) \left(\tau ^2+1\right)^2 (\lambda +3 \mu )^2  (n-1)!^2}((\lambda  (n+2)+\mu  (n+4))^2\notag\\
		&\quad\quad+(n (\lambda +\mu )-2 \mu )^2),\notag\\
		&\zeta _{4,n}=\frac{-2 (n-1) n \left(\tau ^2-1\right) (\lambda +\mu )}{\left(n^2-1\right) \tau ^2 (\lambda +\mu )-(n+1) (\lambda  (n+1)+\mu  (n+5))}=-2+\mathcal{O}\left(\frac{1}{n}\right),\notag\\
		&\zeta _{5,n}=\frac{\left(\lambda ^2 (n^2+2n+2)+2 \lambda  \mu  (n (n+2)+4)+\mu ^2 (n^2+2n+10)\right)}{(n+2) \left((n+1) (\lambda  (n+1)+\mu  (n+5))-\left(n^2-1\right) \tau ^2 (\lambda +\mu )\right)^2}\notag\\
		&\quad\quad\times (n-1)^2 n \left(\tau ^2-1\right)^2=1+\mathcal{O}\left(\frac{1}{n}\right),\notag
	\end{align}
	By substituting the values of \(\zeta _{4,n}\) and \(\zeta _{5,n}\), we can derive	
	\begin{align}\label{eq:zetar}
		1+\zeta _{4,n}+\zeta _{5,n}=\frac{\mathcal{R}_{1}}{n^2}+\mathcal{O}\left(\frac{1}{n^3}\right),
	\end{align}	
	where
	\begin{align}\label{eq:R1}
		\mathcal{R}_{1}=\frac{\lambda ^2 \left(3 \tau ^4-10 \tau ^2+11\right)+2 \lambda  \mu  \left(5 \tau ^4-18 \tau ^2+25\right)+\mu ^2 \left(11 \tau ^4-34 \tau ^2+59\right)}{\left(\tau ^2-1\right)^2 (\lambda +\mu )^2}.
	\end{align}
	Consequently, based on \eqref{eq:zetau} and \eqref{eq:zetar}, one can obtain 
	\begin{align}
		\frac{1+\zeta _{4,n} \gamma_1 ^{2}+\zeta _{5,n} \gamma_1 ^{4}}{1+\zeta _{4,n}+\zeta _{5,n}}&=\frac{1+\left(-2+\mathcal{O}\left(\frac{1}{n}\right)\right)\gamma_{1}^{2}+\left(1+\mathcal{O}\left(\frac{1}{n}\right)\right)\gamma_{1}^{4}}{\frac{\mathcal{R}_{1}}{n^2}+\mathcal{O}(\frac{1}{n^3})}\notag\\
		&=\frac{n^2}{\mathcal{R}_{1}}\left(1-2\gamma_{1}^2+\gamma_{1}^4+\mathcal{O}\left(\frac{1}{n}\right)\right)\cdot\left(1+\mathcal{O}\left(\frac{1}{n}\right)\right)\notag\\
		&=\frac{1-2\gamma_{1}^2+\gamma_{1}^4}{\mathcal{R}_{1}} n^2+\mathcal{O}\left(n\right)<2 \times \frac{1-2\gamma_{1}^2+\gamma_{1}^4}{\mathcal{R}_{1}} n^2,\notag
	\end{align}  
	where $K_{1} = \frac{2\left(1 - 2\gamma_{1}^2 + \gamma_{1}^4\right)}{\mathcal{R}_{1}}$ is defined in \eqref{eq:lemk1k2}. According to Lemma \ref{lem:thmxy}, we obtain
	\[\frac{1+\zeta _{4,n} \gamma_1 ^{2}+\zeta _{5,n} \gamma_1 ^{4}}{1+\zeta _{4,n}+\zeta _{5,n}}<K_{1}n^2\leq\gamma_1^{-n}.\]
	Consequently, the energy distribution of the total internal field on the boundary can be evaluated, and direct calculations show that the result holds when $n$ satisfies \eqref{eq:condition},
	\begin{align}
		\frac{\|\mathbf{u}\|_{L^2(D \setminus \mathcal{N}_{D,1-\gamma_{1}}^{-})^2}^2}{\|\mathbf{u}\|_{L^2(D)^2}^2}=
		&\frac{\zeta _{1,n}(1 +\zeta _{4,n} \gamma_1 ^{2}+\zeta _{5,n} \gamma_1 ^{4})\omega ^{2n-2}\gamma_1 ^{2n}\left(1+\mathcal{O}( \omega^{2})\right)}{\zeta _{1,n}(1 +\zeta _{4,n} +\zeta _{5,n} )\omega ^{2n-2}\left(1+\mathcal{O}(\omega^{2})\right)}\notag\\
		\leq&\gamma_1^{n}\left(1+\mathcal{O}\left(\omega^2\right)\right)\left(1-\mathcal{O}\left(\omega^2\right)\right)\leq\varepsilon+\mathcal{O}\left(\varepsilon\omega^2\right)\ll1.\notag
	\end{align}
	Then, we can conclude that the energy of the internal scattering field $\bu$ in the elastic scattering problem \eqref{eq:xtm} is predominantly localized on the inner surface of $D$.
	With a calculation similar to \eqref{eq:x12}, one can obtain the following result for the coefficients in \eqref{eq:phizh}
	\begin{align}\label{eq:phi212}
		\varphi _{2,1,n}=&\frac{\kappa2^{3-n} \left(\tau ^2-1\right)^2 (\lambda +2 \mu ) \mu ^{\frac{3}{2}-\frac{n}{2}} \rho ^{\frac{n}{2}-\frac{1}{2}}}{\left(\tau ^2+1\right) (\lambda +3 \mu ) (n-2)!}\omega ^{n-1}+q_{3}\omega ^{n+1}+ \mathcal{O}(\omega^{n+3}),\\
		\varphi _{2,2,n}=&-\frac{\mathrm{i} \kappa 2^{3-n} \left(\tau ^2-1\right) (\lambda +2 \mu ) \mu ^{\frac{3}{2}-\frac{n}{2}} \rho ^{\frac{n}{2}-\frac{1}{2}}}{(\lambda +3 \mu ) (n-2)!}\omega ^{n-1}+q_{4}\omega ^{n+1}+ \mathcal{O}(\omega^{n+3}), \notag
	\end{align}
	where
	\begin{align}
		q_{3}=\mathcal{O}\left(\frac{\kappa 2^{1-n}  \mu ^{\frac{1}{2}-\frac{n}{2}} \rho ^{\frac{n}{2}+\frac{1}{2}}}{(n-1)!}\right),\quad q_{4}=\mathcal{O}\left(\frac{\mathrm{i} \kappa 2^{1-n}  \mu ^{\frac{1}{2}-\frac{n}{2}} \rho ^{\frac{n}{2}+\frac{1}{2}}}{(n-1)!}\right).\notag
	\end{align}
	Based on \eqref{eq:sol}, by substituting the coefficients from \eqref{eq:phi212} into formula \eqref{eq:phizh} and incorporating  \eqref{lem:qninutw} from Lemma \ref{lem:sinlay}, we obtain 
	\begin{align}
		\bu^s =&S_{\partial D }^{\omega }[\bm{\varphi} _{2} ] (\bx) 
		=\varphi _{2,1,n}S_{\partial D }^{\omega }[e^{\mathrm{i}n\theta }\bm{\upsilon}] (\bx)+\varphi _{2,2,n}S_{\partial D }^{\omega }[e^{\mathrm{i}n\theta }\mathbf{t}] (\bx) \label{eq:us}\\
		=&\varphi _{2,1,n} \frac{-\mathrm{i}\pi}{4\omega^{2}\rho }\big(nk_{s}J_{n}(k_{s})\bm{\Psi}_{n}^{s,o}(k_{s}|\mathbf{x}|)+k_{p}^{2}J_{n}^{\prime}\big(k_{p}\big)\bm{\Psi}_{n}^{p,o}(k_{p}|\mathbf{x}|)\big) \notag\\
		&+\varphi _{2,2,n} \frac{-\pi}{4\omega^{2}\rho }\big(k_{s}^{2}J_{n}^{\prime}(k_{s})\bm{\Psi}_{n}^{s,o}(k_{s}|\mathbf{x}|)+nk_{p}J_{n}(k_{p})\bm{\Psi}_{n}^{p,o}(k_{p}|\mathbf{x}|)\big).\notag
	\end{align}
	Subsequently, by substituting \eqref{eq:qp0} from Lemma \ref{lem:sinlay} into \eqref{eq:us}, applying the asymptotic expansions in \eqref{eq:jn}, \eqref{eq:hn}, and \eqref{eq:dtgx}, and carrying out the necessary calculations, one can obtain that $\bu^s$ have the following asymptotic expansion: 
	\begin{align}
		\bu^{s}
		=&\left(\left(\xi _5|\bx|^{-n-1}+\xi_6|\bx|^{-n+1}\right)e^{\mathrm{i}n\theta }\bm{\upsilon}+\left(\xi _7|\bx|^{-n-1}+\xi_8|\bx|^{-n+1}\right)e^{\mathrm{i}n\theta }\bm{t}\right)\omega ^{n-1}\notag\\
		&\times \left(1+\mathcal{O}\left(\omega^2\right)\right),\notag
	\end{align}
	where
	\begin{align}
		\xi _5&=-\frac{\kappa 2^{1-n} \mu ^{\frac{1}{2}-\frac{n}{2}} \rho ^{\frac{n}{2}-\frac{1}{2}} (n-1) \left(\tau ^2-1\right)  \left(\tau ^2 (\lambda +3 \mu )+(n+1) (\lambda +\mu )\right)}{(n+1) \left(\tau ^2+1\right) (\lambda +3 \mu )  (n-1)!},\notag\\
		\xi _6&=\frac{\kappa 2^{1-n} \mu ^{\frac{1}{2}-\frac{n}{2}} \rho ^{\frac{n}{2}-\frac{1}{2}} \left(\tau ^2-1\right) (2 \mu +n (\lambda +\mu ))}{\left(\tau ^2+1\right) (\lambda +3 \mu )  (n-1)!},\notag\\
		\xi _7&=\frac{\mathrm{i} \kappa 2^{1-n}\mu ^{\frac{1}{2}-\frac{n}{2}} \rho ^{\frac{n}{2}-\frac{1}{2}} (n-1) \left(\tau ^2-1\right)  \left(\tau ^2 (\lambda +3 \mu )+(n+1) (\lambda +\mu )\right)}{(n+1) \left(\tau ^2+1\right) (\lambda +3 \mu ) (n-1)!},\notag\\
		\xi _8&=-\frac{\mathrm{i}  \kappa 2^{1-n}\mu ^{\frac{1}{2}-\frac{n}{2}} \rho ^{\frac{n}{2}-\frac{1}{2}} \left(\tau ^2-1\right)  (-2 (\lambda +2 \mu )+n (\lambda +\mu ))}{\left(\tau ^2+1\right) (\lambda +3 \mu ) (n-1)!}.\notag
	\end{align}
	Additionally, through rigorous computations, one has
	\begin{align}
		\xi _5&+\xi _6+\xi _7+\xi _8=\frac{(1+\mathrm{i} ) \kappa 2^{1-n} \left(\tau ^2-1\right) \mu ^{\frac{1}{2}-\frac{n}{2}} \rho ^{\frac{n}{2}-\frac{1}{2}} \left(\mathrm{i}  (n-1) \tau ^2+n+1\right)}{(n+1) \left(\tau ^2+1\right)(n-1)!}.\notag
	\end{align}
	Similar to \eqref{eq:unorm}, we can obtain
	\begin{align}
		&\left\|\bu^{s}\right\|_{L^{2}\left(\left(D_R \setminus \overline{D}\right) \setminus \mathcal{N}_{D,\gamma_{2}-1}^{+}\right)^{2}}^{2}\notag\\
		=&\left(\zeta _{6,n} (\gamma_2 ^{-2n}-R^{-2n})+\zeta _{7,n} (\gamma_2 ^{-2n+2}-R^{-2n+2})+\zeta _{8,n} \left( \gamma_2 ^{-2n+4}-R^{-2n+4}\right)\right)\notag\\
		& \times  \omega ^{2n-2}\left(1+ \mathcal{O}(\omega^{2})\right)\notag\\
		=&\zeta _{6,n}\frac{R^{2n}-\gamma_2 ^{2n}}{\gamma_2 ^{2n}R^{2n}}\left(1+\zeta _{9,n} \frac{R^{2n-2}-\gamma_{2}^{2n-2}}{R^{2n}-\gamma_2 ^{2n}}\gamma_2 ^{2}R^{2}+\zeta _{10,n} \frac{R^{2n-4}-\gamma_{2}^{2n-4}}{R^{2n}-\gamma_2 ^{2n}}\gamma_2 ^{4}R^{4}\right)\notag\\
		&\times \omega ^{2n-2}\left(1+ \mathcal{O}(\omega^{2})\right),\notag
	\end{align}
	where
	\begin{align}
		&\zeta _{6,n}=\frac{\pi\kappa^2 2^{3-2 n} \left(\tau ^2-1\right)^2 \mu ^{1-n} \rho ^{n-1} \left(\tau ^2 (\lambda +3 \mu )+(n+1) (\lambda +\mu )\right)^2}{n(n+1)^2 \left(\tau ^2+1\right)^2 (\lambda +3 \mu )^2  (n-2)!^2},\notag\\
		&\zeta_{7,n}=-\frac{\pi\kappa^2 4^{2-n} \left(\tau ^2-1\right)^2 (\lambda +\mu ) \mu ^{1-n} \rho ^{n-1} \left(\tau ^2 (\lambda +3 \mu )+(n+1) (\lambda +\mu )\right)}{(n-1)(n+1) \left(\tau ^2+1\right)^2 (\lambda +3 \mu )^2  (n-2)!^2},\notag\\
		&\zeta_{8,n}=\frac{\pi\kappa^2 4^{1-n} \left(\tau ^2-1\right)^2 \rho ^{n-1} \left((\lambda  (n-2)+\mu  (n-4))^2+(2 \mu +n (\lambda +\mu ))^2\right)}{\mu ^{n-1} (n-2)\left(\tau ^2+1\right)^2 (\lambda +3 \mu )^2 (n-1)!^2},\notag\\
		&\zeta _{9,n}=\frac{-2 n (n+1) (\lambda +\mu )}{\left(n^2-1\right) (\lambda +\mu )+(n-1) \tau ^2 (\lambda +3 \mu )}=-2+\frac{\zeta_{\frac{1}{n}}}{n}+\mathcal{O}\left(\frac{1}{n^2}\right),\notag\\
		&\zeta _{10,n}=\frac{n (n+1)^2 \left((\lambda  (n-2)+\mu  (n-4))^2+(2 \mu +n (\lambda +\mu ))^2\right)}{2 (n-2) (n-1)^2 \left(\tau ^2 (\lambda +3 \mu )+(n+1) (\lambda +\mu )\right)^2}\notag\\
		&\quad\quad=1-\frac{\zeta_{\frac{1}{n}}}{n}+\mathcal{O}\left(\frac{1}{n^2}\right),\notag\\
		&\zeta_{\frac{1}{n}}=\frac{2\tau ^2 (\lambda +3 \mu )-2\lambda -2\mu}{(\lambda +\mu )}.\notag
	\end{align}
	Through rigorous calculations, we can derive	
	\begin{align}\label{eq:zeta2}
		\zeta _{6,n}+\zeta _{7,n}+\zeta _{8,n}=\mathcal{O}\left(\frac{4^{1-n} \mu ^{1-n} \rho ^{n-1}}{n  (n-1)!^2}\right),\quad
		1+\zeta _{9,n}+\zeta _{10,n}=\frac{\mathcal{R}_{2}}{n^2}+\mathcal{O}\left(\frac{1}{n^3}\right),
	\end{align}
	where
	\begin{align}\label{eq:R2}
		\mathcal{R}_{2}=&\frac{3 \lambda ^2+\tau ^4 (\lambda +3 \mu )^2-2 \tau ^2 (\lambda +\mu ) (\lambda +3 \mu )+10 \lambda  \mu +11 \mu ^2}{(\lambda +\mu )^2}.
	\end{align}
	
	Finally, we show that the external scattered field $\bu^s$ is also concentrated, in terms of energy, on the exterior surface of the domain $D$,
	\begin{align}
		&\frac{\|\mathbf{u}^s\|_{L^2\left(\left(D_R \setminus \overline{D}\right) \setminus \mathcal{N}_{D,\gamma_{2}-1}^{+}\right)^2}^2}{\|\mathbf{u}^s\|_{L^2\left(D_{R}\setminus \overline{D}\right)^2}^2}\notag\\
		=&\Big\{\frac{1-2\gamma_{2}^2+\gamma_{2}^4}{\mathcal{R}_{2}}\times\frac{n^2}{\gamma_{2}^{2n}}+\mathcal{O}\left(\frac{n}{\gamma_{2}^{2n}}\right)\Big\} \left(1+\mathcal{O}\left(\omega^2\right)\right)\left(1-\mathcal{O}\left(\omega^2\right)\right)\notag\\
		\leq&K_{2}\times\frac{n^2}{\gamma_{2}^{2n}} \left(1+\mathcal{O}\left(\omega^2\right)\right)\left(1-\mathcal{O}\left(\omega^2\right)\right)\notag\\
		\leq&\gamma_2^{-n}\left(1+\mathcal{O}\left(\omega^2\right)\right)\left(1-\mathcal{O}\left(\omega^2\right)\right)\leq\varepsilon+\mathcal{O}\left(\varepsilon\omega^2\right)\ll1,\notag
	\end{align}
	where $K_{2} =\frac{2\left(1-2\gamma_{2}^2+\gamma_{2}^4\right)}{\mathcal{R}_{2}}$ is defined in \eqref{eq:lemk1k2}. According to Lemma \ref{lem:thmxy}, the inequality $K_{2}n^2 \leq \gamma_{2}^{n}$ holds.
	
	The proof is complete.
\end{proof}
\begin{rem}
	In this paper, to streamline our theoretical presentation, we specialize our analysis to the case where the material $D$ embedded in the soft elastic background is a unit disk. While we believe that our theoretical findings can be extended to high-contrast materials of more general geometries, a comprehensive investigation of such generalizations lies beyond the scope of the current study and will be addressed in future research. Furthermore, as indicated by Theorem \ref{thm:isl}, the scattered wave exhibits boundary localization in the vicinity of \(\partial D\). This implies that, with an appropriately chosen incident wave, the material \(D\) achieves near-cloaking under exterior measurements.
\end{rem}

Theorem \ref{thm:isl} establishes that for any given material parameter $\delta$ and  $\varepsilon$, the boundary localization can be realized in both the interior total field $\mathbf{u}|_D$ and exterior scattered field $\mathbf{u}^s|_{\mathbb{R}^2\setminus\overline{D}}$ by selecting the incident wave index $n$ according to condition~\eqref{eq:condition}. Moreover, Proposition \ref{cor:isl} demonstrates that for a given boundary localization level $\varepsilon$, by selecting an appropriate index $n$ of the incident wave $\mathbf{u}^i$, the boundary localization phenomenon for both the interior total field $\mathbf{u}|_D$ and the exterior scattered field $\mathbf{u}^s|_{\mathbb{R}^2 \setminus \overline{D}}$ can be realized through optimal adjustment of the material parameter $\delta$ with respect to $\varepsilon$, $\gamma_1$, and $\gamma_2$.
\begin{prop}\label{cor:isl}
	Under the same assumptions as  Theorem \ref{thm:isl}, for any fixed parameters $\gamma_1 \in (0,1)$ and $\gamma_2 \in (1,R)$, and a fixed boundary localization level $\varepsilon \ll 1$ , if the material parameter $\delta$ satisfies
	\begin{align}\label{eq:327}
		\delta \leq \delta_{0}=\min\left(\frac{\ln\gamma_{1}}{\ln\varepsilon},-\frac{\ln\gamma_{2}}{\ln\varepsilon}\right),
	\end{align}
	and if the index $n$ associated with the incident wave $\mathbf{u}^i$ defined in \eqref{eq:ui} satisfies
	\begin{align}\label{eq:n27}
		n\geq\max\left({n_{2},n_{4}, \frac{1}{\delta}}\right),
	\end{align}
	where $n_{2}$ and $n_{4}$ are given in \eqref{eq:n24}, then the interior total field $\mathbf{u}|_D$ and the exterior scattered field $\mathbf{u}^s|_{\mathbb{R}^2 \setminus \overline{D}}$ exhibits the boundary localization.
\end{prop}

\begin{proof}
	Recall that $n_1$ and $n_3$ are defined in \eqref{eq:n13}. When \eqref{eq:327} is satisfied, it can be verified that 
	\begin{align}\label{eq:326}
		\frac{1}{\delta} \geq \max\{n_1, n_3\}. 
	\end{align}
	Combining \eqref{eq:n27} and \eqref{eq:326}, we deduce that the index $n$ satisfies the condition \eqref{eq:condition} in Theorem \ref{thm:isl}. Consequently,  Proposition \ref{cor:isl} can be proved directly from  Theorem \ref{thm:isl}. 
\end{proof}

\section{Surface resonances and stress concentrations}
\label{sec:experiments}
In  Theorem \ref{thm:sre}, we establish that the internal total field $\mathbf{u}|_D$ and the external scattered field $\mathbf{u}^s|_{\mathbb{R}^2 \setminus \overline{D}}$ exhibit quasi-Minnaert resonance in the sense of Definition \ref{def:quasi minnaert}. The quasi-Minnaert resonance is characterized by two fundamental phenomena in the wave fields generated by the hard material:
boundary localization (Definition \ref{def:surface localized}) and surface resonance (Definition \ref{def:surface resonant}). These dual effects collectively characterize the distinctive wave interaction patterns arising from the elastic properties of the material within the sub-wavelength regime. Furthermore,  Theorem \ref{thm:eu} demonstrates that the fields $\mathbf{u}|_D$ and $\mathbf{u}^s|_{\mathbb{R}^2 \setminus \overline{D}}$ associated with the scattering problem \eqref{eq:xtm}, where  $\mathbf u^i$ is given by \eqref{eq:ui} and  \( D \) is a unit disk, generating strong stress concentrations near $\partial D$.
\begin{thm}\label{thm:sre}
	Consider the elastic scattering problem \eqref{eq:xtm}. Let $(D;\tilde{\lambda},\tilde{\mu},\tilde{\rho})$ represent the unit disk embedded within a homogeneous elastic medium $(\mathbb{R}^2\setminus\overline{D};\lambda,\mu,\rho)$ in $\mathbb{R}^2$. Recall that $\mathcal{N}_{D,1-\gamma_{1}}^{-}$ and $\mathcal{N}_{D,\gamma_{2}-1}^{+}$ are defined in  \eqref{eq:Mdef}. Under the assumptions \eqref{eq:hicon}-\eqref{eq:de}, \eqref{eq:lambda=} and the parameters $\varepsilon \ll 1,$ if the incident wave index $n$ satisfies 
	\begin{align}\notag
		n\geq \max{\left(n_1,n_2,n_3,n_4, 1/\delta^2\right)},
	\end{align}
	then the corresponding total field ${\bf u}$ in $D$ and the scattered wave field $\mathbf{u}^s$ in $\mathbb R^2 \backslash \overline{D} $ satisfy
	\begin{align}\label{eq:grath1}
		\frac{\|\nabla\mathbf{u}\|_{L^2(\mathcal{N}_{D,1-\gamma_{1}}^{-})^2}}{\|{\bu}^i\|_{L^2(D)^2}}=\mathcal{O}\left(n\delta\right)\gg1,\quad
		\frac{\|\nabla\bu^s\|_{L^2(\mathcal{N}_{D,\gamma_{2}-1}^{+})^2}}{\|\bu^i\|_{L^2(D)^2}}=\mathcal{O}\left({n}\right)\gg1.
	\end{align}
	Namely, the quasi-Minnaert resonance occurs.
\end{thm}
\begin{proof}
	First, we prove the first equality in \eqref{eq:grath1}. To simplify the discussion in this section, we denote $r=|\bx|$, and let  $\hat{r}$, $\hat{\theta}$ replace $\bm{\upsilon}$, $\bm{t}$. This is followed by the introduction of the following additional terms:
	\begin{align}
		a_{11} &=\varphi _{1,1,n}\frac{-2\mathrm{i} \pi   n^2 \tilde{k}_{s}  H_n\left(\tilde{k}_{s}  \right)}{4 \omega ^2\tilde{\rho}},\qquad 
		&a_{12} &=\varphi  _{1,2,n}\frac{-2\pi   n \tilde{k}_{s} ^2  H_n'\left(\tilde{k}_{s} \right)}{4 \omega ^2\tilde{\rho}},\label{eq:ure}\\
		a_{13}&=\varphi  _{1,1,n}\frac{-2\mathrm{i} \pi   \tilde{k}_{p} ^2  H_n'\left(\tilde{k}_{p} \right)}{4 \omega ^2\tilde{\rho}},\qquad 
		&a_{14}&=\varphi  _{1,2,n}\frac{-2\pi   n \tilde{k}_{p}   H_n\left(\tilde{k}_{p} \right)}{4 \omega ^2\tilde{\rho}} ,\notag \\
		a_{21}&=\varphi  _{1,1,n}\frac{2 \pi   n \tilde{k}_{s}  H_n\left(\tilde{k}_{s}  \right)}{4 \omega ^2\tilde{\rho}},\qquad
		&a_{22}&=\varphi  _{1,2,n}\frac{-2\mathrm{i}\pi    \tilde{k}_{s} ^2  H_n'\left(\tilde{k}_{s} \right)}{4 \omega ^2\tilde{\rho}},\notag\\
		a_{23}&=\varphi  _{1,1,n}\frac{2n \pi   \tilde{k}_{p} ^2  H_n'\left(\tilde{k}_{p} \right)}{4 \omega ^2\tilde{\rho}},\qquad
		&a_{24}&=\varphi  _{1,2,n}\frac{-2\mathrm{i}\pi   n^2 \tilde{k}_{p}   H_n\left(\tilde{k}_{p} \right)}{4 \omega ^2\tilde{\rho}}.\notag
	\end{align}
	
	Accordingly, when $n$ is sufficiently large, $\bu$ can be expressed as \eqref{eq:usw}. Moreover, by utilizing the notation introduced in \eqref{eq:ure}, it can be expressed equivalently as
	\begin{align}
		\bu =&a_{11}\frac{J_{n}(\tilde{k}_{s}r) }{\tilde{k}_{s}r} e^{\mathrm{i}n\theta }\hat{r} +a_{12}\frac{J_{n}(\tilde{k}_{s}r) }{\tilde{k}_{s}r} e^{\mathrm{i}n\theta }\hat{r}+a_{13}J_{n}'(\tilde{k}_{p}r)e^{\mathrm{i}n\theta }\hat{r}   +a_{14}J_{n}'(\tilde{k_{p}}r)e^{\mathrm{i}n\theta }\hat{r} \notag \\
		&+a_{21}J_{n}'(\tilde{k}_{s}r)e^{\mathrm{i}n\theta }\hat{\theta}+a_{22}J_{n}'(\tilde{k_{s}}r)e^{\mathrm{i}n\theta }\hat{\theta}+a_{23}\frac{J_{n}(\tilde{k_{p}}r) }{\tilde{k_{p}}r} e^{\mathrm{i}n\theta }\hat{\theta}+a_{24}\frac{J_{n}(\tilde{k_{p}}r) }{\tilde{k_{p}}r} e^{\mathrm{i}n\theta }\hat{\theta}.\label{eq:uret}
	\end{align}
			Through complex and tedious calculations, we can obtain
			\begin{align}
				\nabla \bu=&(a_{11}+a_{12})\nabla \left ( \frac{J_{n} (\tilde{k}_{s}  r )}{\tilde{k}_{s}r} e^{\mathrm{i}n\theta } \hat{r}  \right )+(a_{13}+a_{14}) \nabla \left(J_{n}' (\tilde{k}_{p}  r ) e^{\mathrm{i}n\theta } \hat{r} \right)+(a_{21}+a_{22}) \notag\\
				&\times\nabla \left(J_{n}' (\tilde{k}_{s}  r ) e^{\mathrm{i}n\theta } \hat{\theta}\right)+(a_{23}+a_{24})\nabla \left ( \frac{J_{n} (\tilde{k}_{p}  r )}{\tilde{k}_{p}r }e^{\mathrm{i}n\theta } \hat{\theta}  \right )\notag\\
				=&A_{11}\hat{r}\otimes \hat{r} +A_{12}\hat{r}\otimes \hat{\theta }+A_{22}\hat{\theta }\otimes \hat{\theta }+A_{21}\hat{\theta }\otimes \hat{r },\label{eq:grau}
			\end{align}
			where
			\begin{align}
				A_{11}=&\frac{\kappa 2^{1-n} \tau ^2 \epsilon  \mu ^{\frac{1}{2}-\frac{n}{2}} \rho ^{\frac{n}{2}-\frac{1}{2}} r^{n-2} e^{\mathrm{i}n \theta  } }{\left(\tau ^2+1\right) (\lambda +3 \mu )  (n-2)!}(\lambda  \left(n \left(r^2-1\right) \left(\tau ^2-1\right)+\tau ^2+1\right),\notag\\
				A_{12}=&\frac{\mathrm{i} \kappa 2^{2-n} \tau ^2 \epsilon  \mu ^{\frac{1}{2}-\frac{n}{2}} \rho ^{\frac{n}{2}-\frac{1}{2}} r^{n-2} e^{\mathrm{i} n\theta  } }{\left(\tau ^2+1\right) (\lambda +3 \mu )  (n-2)!}(\lambda  \left(n \left(r^2-1\right) \left(\tau ^2-1\right)+\tau ^2+1\right)\notag\\
				&+\mu  \left(n \left(r^2-1\right) \left(\tau ^2-1\right)+\tau ^2+5\right))\omega^{n-1}\left(1+\mathcal{O}\left(\omega^2\right)\right),\notag\\
				A_{21}=&0,\notag\\
				A_{22}=&-\frac{\kappa 2^{1-n} \tau ^2 \epsilon  \mu ^{\frac{1}{2}-\frac{n}{2}} \rho ^{\frac{n}{2}-\frac{1}{2}} r^{n-2} e^{\mathrm{i}n \theta  } }{\left(\tau ^2+1\right) (\lambda +3 \mu )  (n-2)!}(\lambda  \left(n \left(r^2-1\right) \left(\tau ^2-1\right)+\tau ^2+1\right)\notag\\
				&+\mu  \left(n \left(r^2-1\right) \left(\tau ^2-1\right)+2 r^2 \left(\tau ^2-1\right)+\tau ^2+5\right))\omega^{n-1}\left(1+\mathcal{O}\left(\omega^2\right)\right).\label{eq:graa}
			\end{align}
			Based on \eqref{eq:grau}, through rigorous computations, we can obtain
			\begin{align}
				&\|\nabla\bu\|_{L^{2}(\mathcal{N}_{D,1-\gamma_{1}}^{-})^2 }^{2} \label{eq:graufs}\\
				=&\left(\mathcal{A}_{4,n}\left(1-\gamma _{1}^{2n-2}\right)+\mathcal{A}_{5,n}\left(1-\gamma _{1}^{2n}\right)+\mathcal{A}_{6,n}\left(1-\gamma _{1}^{2n+2}\right)\right)\omega^{2n-2}\left(1+\mathcal{O}\left(\omega^2\right)\right)\notag\\
				=&\mathcal{A}\omega^{2n-2}\left(1+\mathcal{O}\left(\omega^2\right)\right),\notag		
			\end{align}
			where
			\begin{align}
				\mathcal{A}_{4,n}&=\frac{6\pi\kappa^2 4^{1-n} \tau ^4 \epsilon ^2 \mu ^{1-n} \rho ^{n-1}  \left(\lambda +5 \mu -(n-1) \tau ^2 (\lambda +\mu )+\lambda  n+\mu  n\right)^2}{(n-1)\left(\tau ^2+1\right)^2 (\lambda +3 \mu )^2  (n-2)!^2},\notag\\
				\mathcal{A}_{5,n}&=\frac{12\pi\kappa^2  \tau ^4 \epsilon ^2    n \left(\tau ^2-1\right) (\lambda +\mu ) \left(\lambda +5 \mu -(n-1) \tau ^2 (\lambda +\mu )+\lambda  n+\mu  n\right)}{4^{n-1}\rho ^{1-n}\mu ^{n-1}\left(\tau ^2+1\right)^2 (\lambda +3 \mu )^2 n (n-2)!^2},\notag\\
				\mathcal{A}_{6,n}&=\frac{2\pi\kappa^2  \tau ^4 \epsilon ^2  \rho ^{n-1}  \left(\tau ^2-1\right)^2 \left(4 \mu ^2+3 n^2 (\lambda +\mu )^2\right)}{4^{n-1}\mu ^{n-1}\left(\tau ^2+1\right)^2 (\lambda +3 \mu )^2 (n+1) (n-2)!^2}.\notag
			\end{align}
			Additionally, through rigorous computations, one has
			\begin{align}
				\mathcal{A}_{4,n}+\mathcal{A}_{5,n}+\mathcal{A}_{6,n}=\mathcal{O}\left(\frac{\kappa^2 2^{3-2 n} \tau ^4 \epsilon ^2 \mu ^{1-n} \rho ^{n-1} }{(n-1)!   (n-2)!}\right).\notag
			\end{align}

			Similarly, we can calculate $\|\bu^i\|_{L^2(D)^2}^2$, where $\bu^i$ is defined in \eqref{eq:ui}. After performing the calculations, we obtain the asymptotic form of $\bu^i$,
			\begin{align}
				\bu^i=&\frac{2\kappa nJ_{n}(k_{s}r)}{k_{s}r}e^{\mathrm{i}n\theta}\hat{r}+2\mathrm{i}\kappa J_{n}^{\prime}(k_{s}r)e^{\mathrm{i}n\theta}\hat{\theta}\notag\\
				=&\frac{\kappa     \rho ^{\frac{n}{2}-\frac{1}{2}} r^{n-1} \omega ^{n-1}}{2^{n-1}\mu ^{\frac{n}{2}-\frac{1}{2}}(n-1)!}\left(1+\mathcal{O}(\omega^2)\right)e^{\mathrm{i}n\theta}\hat{r}+\frac{i \kappa    \rho ^{\frac{n}{2}-\frac{1}{2}} r^{n-1} \omega ^{n-1}}{2^{n-1}\mu ^{\frac{n}{2}-\frac{1}{2}}(n-1)!}\left(1+\mathcal{O}(\omega^2)\right)e^{\mathrm{i}n\theta}\hat{\theta}.\notag
			\end{align}
			Furthermore, we can obtain
			\begin{align}
				\|\bu^{i}\|_{L^{2}(D)^2 }^{2}
				&=  \frac{\pi\kappa ^2 2^{3-2 n} \mu ^{1-n} \rho ^{n-1} }{n!(n-1)!}\omega^{2n-2}(1+\mathcal{O}(\omega ^2 )).\label{eq:graufsiu}
			\end{align}
			Combining \eqref{eq:graufs} and \eqref{eq:graufsiu}, one can obtain 
			\begin{align}\notag
				\frac{\|\nabla\mathbf{u}\|_{L^2(\mathcal{N}_{D,1-\gamma_{1}}^{-})^2}^2}{\|{\bu}^i\|_{L^2(D)^2}^2}=\frac{\mathcal{A}\omega^{2n-2}\left(1+\mathcal{O}\left(\omega^2\right)\right)}{ \frac{\pi\kappa ^2 2^{3-2 n} \mu ^{1-n} \rho ^{n-1} }{n!(n-1)!}\omega^{2n-2}(1+\mathcal{O}(\omega ^2 ))}=\mathcal{O}\left(n^2\delta^2\right).
			\end{align}  
			
			Under assumptions \eqref{eq:hicon}-\eqref{eq:de} and \eqref{eq:lambda=}, the  parameter $\delta$ satisfies $\delta\ll1$. Thus, for $n \geq\delta^{-2}$, we can obtain
			\begin{align}\label{eq:surface1} 
				\frac{\|\nabla\mathbf{u}\|_{L^2(\mathcal{N}_{D,1-\gamma_{1}}^{-})^2}}{\|{\bu}^i\|_{L^2(D)^2}}=\mathcal{O}\left(n\delta\right)\gg1.  
			\end{align}  
			From \eqref{eq:surface1}, we establish that the total field $\mathbf{u}$ exhibits surface resonance.  Theorem \ref{thm:isl} demonstrates that for parameters satisfying $\varepsilon \ll 1$ and $n \geq \max\left(n_1, n_2, n_3, n_4\right)$, the total field $\mathbf{u}$ manifests boundary localization. Consequently, these results reveal that the total field $\mathbf{u}$ exhibits quasi-Minnaert resonance.
			
			Next, we establish the second equality in \eqref{eq:grath1}. Combining \eqref{eq:us} and \eqref{eq:qp0}, we can rewrite the scattered field in the following form
				\begin{align}
					\bu^s \notag
					=&\left(a_{31} +a_{32}\right)\frac{H_{n}(k_{s}|\bx|) }{k_{s}|\bx|} e^{\mathrm{i}n\theta }\hat{r}+\left(a_{33} +a_{34}\right)H_{n}'(k_{p}|\bx|)e^{\mathrm{i}n\theta }\hat{r} \notag\\
					&+\left(a_{41}+a_{42}\right)H_{n}'(k_{s}|\bx|)e^{\mathrm{i}n\theta }\hat{\theta}+\left(a_{43}+a_{44}\right)\frac{H_{n}(k_{p}|\bx|) }{k_{p}|\bx|} e^{\mathrm{i}n\theta }\hat{\theta},\notag
				\end{align}
				where
				\begin{align}\label{eq:usxs}
					a_{31} &=\varphi _{2,1,n}\frac{-2\mathrm{i} \pi   n^2 k_{s}  J_n\left(k_{s}  \right)}{4 \omega ^2\rho}, \qquad 
					&a_{32} &=\varphi _{2,2,n}\frac{-2\pi   n k_{s} ^2  J_n'\left(k_{s} \right)}{4 \omega ^2\rho},\\
					a_{33}&=\varphi _{2,1,n}\frac{-2\mathrm{i} \pi   k_{p}^2  j_n'\left(k_{p} \right)}{4 \omega ^2\rho},\qquad 
					&a_{34}&=\varphi _{2,2,n}\frac{-2\pi   n k_{p}   J_n\left(k_{p}\right)}{4 \omega ^2\rho}, \notag \\
					a_{41}&=\varphi _{2,1,n}\frac{2 \pi   n k_{s}  J_n\left(k_{s} \right)}{4 \omega ^2\rho},\qquad
					&a_{42}&=\varphi _{2,2,n}\frac{-2\mathrm{i}\pi    k_{s} ^2  J_n'\left(k_{s} \right)}{4 \omega ^2\rho},\notag\\
					a_{43}&=\varphi _{2,1,n}\frac{2n \pi   k_{p} ^2  J_n'\left(k_{p} \right)}{4 \omega ^2\rho},\qquad
					&a_{44}&=\varphi _{2,2,n}\frac{-2\mathrm{i}\pi   n^2 k_{p}   J_n\left(k_{p} \right)}{4 \omega ^2\rho}.\notag
				\end{align}
				Through complex and tedious calculations, we can obtain
				\begin{align}
					\nabla \bu^{s}=
					&(a_{31}+a_{32})\nabla \left ( \frac{H_{n} (k_{s} r )}{{k_{s}}r} e^{\mathrm{i}n\theta } \hat{r} \right )+(a_{33}+a_{34}) \nabla \left(H_{n}' ({k_{p} } r ) e^{\mathrm{i}n\theta } \hat{r}\right)+(a_{41}+a_{42})\notag \\
					&\nabla \left(H_{n}' ({k_{s} } r ) e^{\mathrm{i}n\theta } \hat{\theta }\right)+(a_{43}+a_{44})\nabla \left ( \frac{H_{n} ({k_{p} } r )}{{k_{p}}r} e^{\mathrm{i}n\theta }\hat{\theta }  \right )\label{eq:graus}\\
					=&A_{31}\hat{r}\otimes \hat{r} +A_{32}\hat{r}\otimes \hat{\theta }+A_{42}\hat{\theta }\otimes \hat{\theta }+A_{41}\hat{\theta }\otimes \hat{r },\notag
				\end{align}
				where
				\begin{align}\label{eq:graxs}
					A_{31}=&\frac{\kappa 2^{1-n} \left(\tau ^2-1\right) \mu ^{\frac{1}{2}-\frac{n}{2}} \rho ^{\frac{n}{2}-\frac{1}{2}} r^{-n-2} e^{\mathrm{i} n\theta  } }{\left(\tau ^2+1\right) (\lambda +3 \mu )  (n-2)!}\Big\{\lambda  \left(-n r^2+n+\tau ^2+1\right)\\
					&+\mu  \left(-n r^2+n-2 r^2+3 \tau ^2+1\right)\Big\}\omega^{n-1}\left(1+\mathcal{O}\left(\omega^2\right)\right),\notag\\
					A_{32}=&-\frac{\mathrm{i} \kappa 2^{2-n} \left(\tau ^2-1\right) \mu ^{\frac{1}{2}-\frac{n}{2}} \rho ^{\frac{n-1}{2}} r^{-n-2} e^{\mathrm{i} n\theta  } }{\left(\tau ^2+1\right) (\lambda +3 \mu )  (n-2)!}\Big\{\tau ^2 (\lambda +3 \mu )\notag\\
					&-(\lambda +\mu ) \left(n \left(r^2-1\right)-1\right)\Big\}\omega^{n-1}\left(1+\mathcal{O}\left(\omega^2\right)\right),\notag\\
					A_{42}=&-\frac{\kappa 2^{1-n} \left(\tau ^2-1\right) \mu ^{\frac{1}{2}-\frac{n}{2}} \rho ^{\frac{n-1}{2}} r^{-n-2} e^{\mathrm{i}n \theta  } }{\left(\tau ^2+1\right) (\lambda +3 \mu )  (n-2)!}\Big\{\lambda  \left(-n r^2+n+\tau ^2+1\right)\notag\\
					&+\mu  \left(-n r^2+n+2 r^2+3 \tau ^2+1\right)\Big\}\omega^{n-1}\left(1+\mathcal{O}\left(\omega^2\right)\right),\notag\\
					A_{41}=&0.\notag
				\end{align}
				Based on \eqref{eq:graus} and \eqref{eq:graxs}, we can obtain
				\begin{align}\label{eq:graufss}
					&\|\nabla\bu^s\|_{L^{2}(\mathcal{N}_{D,\gamma_{2}-1}^{+})^2 }^{2}\\ 
					=&\left(\mathcal{B}_{4,n}\left(1-\frac{1}{\gamma _{2}^{2n+2}}\right)+\mathcal{B}_{5,n}\left(1-\frac{1}{\gamma _{2}^{2n}}\right)+\mathcal{B}_{6,n}\left(1-\frac{1}{\gamma _{2}^{2n-2}}\right)\right)\omega ^{2n-2}\left(1+\mathcal{O}(\omega^2  )\right)	\notag\\
					=&\mathcal{B}\omega ^{2n-2}\left(1+\mathcal{O}(\omega^2  )\right),\notag
				\end{align}
				where
				\begin{align}
					\mathcal{B}_{4,n}&=\frac{3\pi \kappa^2 2^{3-2 n} \left(\tau ^2-1\right)^2 \mu ^{1-n} \rho ^{n-1}\left(\tau ^2 (\lambda +3 \mu )+(n+1) (\lambda +\mu )\right)^2 }{\left(\tau ^2+1\right)^2 (\lambda +3 \mu )^2 (n+1) (n-2)!^2},\notag \\
					\mathcal{B}_{5,n}&=-\frac{6 \pi\kappa^2 2^{3-2 n} \left(\tau ^2-1\right)^2  \rho ^{n-1} n (\lambda +\mu ) \left(\tau ^2 (\lambda +3 \mu )+(n+1) (\lambda +\mu )\right) }{\mu ^{n-1}\left(\tau ^2+1\right)^2 (\lambda +3 \mu )^2  n(n-2)!^2},\notag\\
					\mathcal{B}_{6,n}&=\frac{ \pi\kappa^2 2^{3-2 n} \left(\tau ^2-1\right)^2 \rho ^{n-1}\left(4 \mu ^2+3 n^2 (\lambda +\mu )^2\right) }{\mu ^{n-1} \left(\tau ^2+1\right)^2 (\lambda +3 \mu )^2  (n-1)!(n-2)!},\notag\\
					\mathcal{B}&=\left(\mathcal{B}_{4,n}\left(1-\frac{1}{\gamma _{2}^{2n+2}}\right)+\mathcal{B}_{5,n}\left(1-\frac{1}{\gamma _{2}^{2n}}\right)+\mathcal{B}_{6,n}\left(1-\frac{1}{\gamma _{2}^{2n-2}}\right)\right).\notag   
				\end{align}
				Since $\lambda$, $\mu$, and $\rho$ are all of order $\mathcal{O}(1)$, rigorous calculations yield
				\begin{align}\notag
					\mathcal{B}_{4,n}+\mathcal{B}_{5,n}+\mathcal{B}_{6,n}=\mathcal{O}\left(\frac{  \kappa^2 2^{3-2 n}  \mu ^{1-n} \rho ^{n-1}  }{  (n-1)! (n-2)!}\right).\notag
				\end{align}
				
				When $n$ is sufficiently large, combining \eqref{eq:graufss} and \eqref{eq:graufsiu}, we obtain
				\begin{align}
					\frac{\|\nabla\bu^s\|_{L^2(\mathcal{N}_{D,\gamma_{2}-1}^{+})^2}^2}{\|\bu^i\|_{L^2(D)^2}^2}=&\frac{\mathcal{B}\omega ^{2n-2}(1+\mathcal{O}(\omega^2  ))}{\frac{\pi\kappa ^2 2^{3-2 n} \mu ^{1-n} \rho ^{n-1} }{n!(n-1)!} \omega^{2n-2}(1+\mathcal{O}(\omega^2  ))} =\mathcal{O}\left(n^2\right).\notag
				\end{align}
				
				Theorem~\ref{thm:isl} establishes that when $\varepsilon \ll 1$ and $n \geq \max(n_1, n_2, n_3, n_4)$, the index $n$ is sufficiently large. Furthermore, under assumptions \eqref{eq:hicon}-\eqref{eq:de} and \eqref{eq:lambda=}, the  parameter $\delta$ satisfies $\delta \ll 1$. Consequently, for  $n \geq \delta^{-2}$, we obtain the sufficient largeness condition for $n$. Thus, one has
				\begin{align}\label{eq:surface2} 
					\frac{\|\nabla\bu^s\|_{L^2(\mathcal{N}_{D,\gamma_{2}-1}^{+})^2}}{\|\bu^i\|_{L^2(D)^2}}=\mathcal{O}\left({n}\right)\gg1.  
				\end{align}   
				From \eqref{eq:surface2}, we conclude that the scattered field \(\mathbf{u}^s\) exhibits surface resonance.  Theorem \ref{thm:isl} establishes that when the parameters satisfy \(\varepsilon \ll 1\), \(n \geq \max\left(n_1, n_2, n_3, n_4\right)\), the scattered field \(\mathbf{u}^s\) exhibits boundary localization. Consequently, these results collectively demonstrate that the scattered field \(\mathbf{u}^s\) exhibits quasi-Minnaert resonance.
				
				The proof is complete.
			\end{proof}
			
			 Theorem \ref{thm:sre} establishes that selecting the incident wave index $n$ from \eqref{eq:ui} with \( n \geq \max(n_1, \\n_2, n_3, n_4, \delta^{-2}) \) guarantees the simultaneous occurrence of boundary localization and surface resonance, thereby inducing quasi-Minnaert resonance. Building upon \eqref{eq:surface1} and \eqref{eq:surface2}, Proposition \ref{prop1} demonstrates that by selecting the incident wave index $n$ satisfying $n \geq \delta^{-2}$, surface resonance can occur independently.
			
			\begin{prop}\label{prop1}
				Consider the elastic scattering problem \eqref{eq:xtm}, under the same assumptions as  Theorem \ref{thm:sre}, if $n \geq 1/\delta^2$, then  $\mathbf{u}|_{D}$ and $\mathbf{u}^s|_{\mathbb{R}^2 \setminus \overline{{D}}}$ are surface resonance.
			\end{prop}
			\begin{proof}
				From the proof of Theorem \ref{thm:sre}, we can obtain 
				\begin{align}\notag  
					\frac{\|\nabla\mathbf{u}\|_{L^2(\mathcal{N}_{D,1-\gamma_{1}}^{-})^2}}{\|{\bu}^i\|_{L^2(D)^2}}=\mathcal{O}\left(n\delta\right), \quad\frac{\|\nabla\bu^s\|_{L^2(\mathcal{N}_{D,\gamma_{2}-1}^{+})^2}}{\|\bu^i\|_{L^2(D)^2}}=\mathcal{O}\left({n}\right).  
				\end{align}
				Under assumptions \eqref{eq:hicon}-\eqref{eq:de} and \eqref{eq:lambda=}, the  parameter $\delta$ satisfies $\delta \ll 1$. Consequently, for  $n \geq\delta^{-2}$, we can obtain $\mathcal{O}\left({n}\right)\gg1$ and $\mathcal{O}\left(n\delta\right)\gg1$.
				It follows directly that $\mathbf{u}|_{D}$ and $\mathbf{u}^s|_{\mathbb{R}^2 \setminus \overline{{D}}}$ are surface resonant, where \(\delta\) represents the high contrast ratio between the Lam\'e parameters of the hard material \(D\) and the soft elastic background \(\mathbb{R}^2 \setminus \overline{D}\). Both the internal wave field \( \bu|_D \) and the exterior scattered wave field \( \bu^s|_{\mathbb{R}^2 \setminus \overline{D}} \) exhibit high oscillations.
			\end{proof}
			
			In  Theorem \ref{thm:sre} and  Proposition \ref{prop1}, under the sub-wavelength regime and for given high contrast material structure parameters, by appropriately selecting the parameter \(n\) defining the incident wave \(\bu^i\) in \eqref{eq:ui}, we demonstrate the occurrence of quasi-Minnaert resonance and surface resonance. Conversely, the occurrence of quasi-Minnaert resonance and surface resonance can inform the design of high contrast material parameters.
			\begin{prop}\label{prop:4.2}
				Consider the elastic scattering problem \eqref{eq:xtm}, under the assumptions of Theorem \ref{thm:sre}	and for boundary localization level \( \varepsilon\ll1 \), the parameter \(\delta\) can be chosen to satisfy 
				\[\delta\leq\delta_{1}=\min\left(\left(\frac{\ln\gamma_{1}}{\ln\varepsilon}\right)^{1/2},\left(-\frac{\ln\gamma_{2}}{\ln\varepsilon}\right)^{1/2}\right),\]
				to design a material structure with desired properties. If the index \(n\) associated with \(\bu^i\) defined in \eqref{eq:ui} satisfies
				$ n\geq\max\left({n_{2},n_{4},1/\delta^2}\right),$
				the designed material structure exhibits quasi-Minnaert resonance.
			\end{prop}
			\begin{proof}
				By combining the proofs of  Theorems \ref{thm:isl} and \ref{thm:sre}, this proof can be completed. Thus, we omit the details here.
			\end{proof}
			
			In  Theorem \ref{thm:eu}, we prove that the internal total field $\mathbf {u}|_D$ is stress concentrated in $\mathcal{N}_{D,1-\gamma_{1}}^{-}$, and the external scattered field $\mathbf u^s|_{\mathbb R^2\setminus\overline D} $ is stress concentrated in $\mathcal{N}_{D,\gamma_{2}-1}^{+} $.
			\begin{thm}\label{thm:eu}
				Consider the elastic scattering problem \eqref{eq:xtm}. Let $(D;\tilde{\lambda},\tilde{\mu},\tilde{\rho})$ be the unit disk embedded in a homogeneous elastic medium $(\mathbb{R}^2\setminus\overline{D};\lambda,\mu,\rho)$ in $\mathbb{R}^2$. We introduce $E(\bu)$ and $E(\bu^s)$ 
				\begin{align}\label{eq:eueus}
					E(\mathbf{u})=\int_{\mathcal{N}_{D,1-\gamma_{1}}^{-}}\sigma(\mathbf{u}):\nabla\overline{\mathbf{u}}\mathrm{d}\bx,\quad E(\mathbf{u}^s)=\int_{\mathcal{N}_{D,\gamma_{2}-1}^{+}}\sigma(\mathbf{u}^s):\nabla\overline{\mathbf{u}^s}\mathrm{d}\bx,
				\end{align}
				where
				\[\sigma(\mathbf{u})=\tilde{\lambda}(\nabla\cdot\mathbf{u})\mathcal{I}+\tilde{\mu}\left(\nabla\mathbf{u}+\nabla\mathbf{u}^\top\right),\quad\sigma(\mathbf{u}^s)=\lambda(\nabla\cdot\mathbf{u}^s)\mathcal{I}+\mu\left(\nabla\mathbf{u}^s+(\nabla\mathbf{u}^s)^\top\right).\]
				Here, \(\tilde{\lambda}\), \(\tilde{\mu}\), $\lambda$ and $\mu$ are defined in \eqref{eq:hicon}, and \(\mathcal{I}\) denotes the identity matrix. Recall that $\mathcal{N}_{D,1-\gamma_{1}}^{-}$ and $\mathcal{N}_{D,\gamma_{2}-1}^{+}$ are defined in  \eqref{eq:Mdef}. The operator “:” denotes the Hadamard product of two matrices. Under the assumptions \eqref{eq:hicon}-\eqref{eq:de}, \eqref{eq:lambda=} and the parameter \(\varepsilon \ll 1\), if the incident wave $\bu^i$ is chosen as defined in \eqref{lem:qni}, and  the incident wave index $n$ satisfies $n \geq  1/\delta$,
				the following condition holds:
				\begin{align}\label{eq:eug}
					\frac{E(\bu)}{\|{\bu}^i\|_{L^2(D)^2}^2} = \mathcal{O}\left(n^2\delta\right) \gg 1,  \quad
					\frac{E(\mathbf{u^{s}})}{\|\bu^i\|_{L^2(D)^2}^2} = \mathcal{O}\left(n^2\right) \gg 1.
				\end{align}
				Consequently, the corresponding total field $\mathbf{u}|_D$ and the scattered field $\mathbf u^s|_{\mathbb R^2\setminus\overline D}$ exhibit significant stress concentration.
			\end{thm}
			\begin{proof}
				The proof of this theorem shall be divided into two parts.
				First, we prove the first equality in \eqref{eq:eug}.
				We select the appropriate incident wave defined in \eqref{eq:ui}, and then we can obtain the asymptotic expansion of $\nabla \bu$ for the interior total field $\bu$ of the scattering problem \eqref{eq:xtm}, as shown in \eqref{eq:grau}. When $n$ is sufficiently large, one has 
				\begin{align}\label{eq:graugrau}
					(\nabla \overline{\bu} )\nabla \bu
					=&(\overline{A_{11} }A_{11}+\overline{A_{12} }A_{21}) (\hat{r} \otimes \hat{r })+(\overline{A_{11} }A_{12}+\overline{A_{12} }A_{22} ) (\hat{r} \otimes \hat{\theta })\\
					&+(\overline{A_{21} }A_{11} +\overline{A_{22} }A_{21} )(\hat{\theta} \otimes \hat{r })+(\overline{A_{21} }A_{12} +\overline{A_{22} }A_{22})(\hat{\theta} \otimes \hat{\theta}),\notag
				\end{align}
				where $A_{11}, A_{12}, A_{21} $ and $A_{22}$ are defined in \eqref{eq:graa}. and based on \eqref{eq:graa}, one can obtain 
				\begin{align}\label{eq:grabar}
					&{\sf Tr}((\nabla \overline{\bu} )\nabla \bu)=(\overline{A_{11} }A_{11}+\overline{A_{12} }A_{21})+(\overline{A_{21} }A_{12} +\overline{A_{22} }A_{22})=|A_{11}|^{2}+|A_{22}|^{2}. 
				\end{align}
				Therefore, by substituting \eqref{eq:ure}, \eqref{eq:graa} and \eqref{eq:grabar} into \eqref{eq:eueus}, one can obtain 
				\begin{align}\label{eq:euz}
					E(\mathbf{u})
					=&-\frac{\lambda }{\delta } \int_{\mathcal{N}_{D,1-\gamma_{1}}^{-}}J_n(\tilde{k}_{p} r)\tilde{k}_{p} e^{\mathrm{i}n\theta }\left(a_{13}+a_{14} \right)\left(\overline{A}_{11}+ \overline{A}_{22}\right)r\mathrm{d}r\mathrm{d}\theta\notag\\
					&+\frac{\mu }{\delta } \int_{\mathcal{N}_{D,1-\gamma_{1}}^{-}}\left({\sf Tr}(\nabla\mathbf{u}^{H} \nabla\mathbf{u})+{\sf Tr}(\nabla\overline{\mathbf{u}}\nabla{\mathbf{u}})\right)r\mathrm{d}r\mathrm{d}\theta\notag\\
					=&\left(\mathcal{H}_{1,n}\left(1-\gamma_1^{2n-2}\right)+\mathcal{H}_{2,n}\left(1-\gamma_1^{2n}\right)+\mathcal{H}_{3,n}\left(1-\gamma_1^{2n+2}\right)\right)\omega ^{2n-2}(1+\mathcal{O}(\omega ^2 ))\notag\\
					=&\mathcal{H}\omega ^{2n-2}(1+\mathcal{O}(\omega ^2 )),
				\end{align}
				where\[\mathcal{H}_{1,n}+\mathcal{H}_{2,n}+\mathcal{H}_{3,n}=\mathcal{O}\left(\frac{\kappa^2 2^{3-2 n} \tau ^4 \epsilon ^2 \mu ^{2-n} \rho ^{n-1}}{\delta  (n-1)!  (n-2)!}\right).\]
				
				By combining \eqref{eq:euz} and \eqref{eq:graufsiu}, when $n$ is sufficiently large, one can obtain 
				\begin{align}
					\frac{E(\bu)}{\|{\bu}^i\|_{L^2(D)^2}^2}=&\frac{\mathcal{H}\omega^{2n-2}\left(1+\mathcal{O}\left(\omega^2\right)\right)}{ \frac{\pi\kappa ^2 2^{3-2 n} \mu ^{1-n} \rho ^{n-1} }{n!(n-1)!}(1-\gamma _{1}^{2n})\omega^{2n-2}(1+\mathcal{O}(\omega ^2 ))}=\mathcal{O}\left(n^2\delta\right).\notag
				\end{align}
				Under assumptions \eqref{eq:hicon}-\eqref{eq:de} and \eqref{eq:lambda=}, the  parameter $\delta$ satisfies $\delta\ll1$. Thus, for $n \geq \delta^{-1}$, we can obtain
				\begin{align}\label{eq:eu1}  
					\frac{E(\bu)}{\|{\bu}^i\|_{L^2(D)^2}^2} = \mathcal{O}\left(n^2\delta\right) \gg 1.  
				\end{align}  
				
				Next, we establish the second equality in \eqref{eq:eug}.
				By applying a similar line of argument as for \eqref{eq:graugrau}, and based on \eqref{eq:graus}, it becomes feasible to obtain the following result when $n$ is sufficiently large,
				\begin{align}\label{eq:grausgraus}
					(\nabla \overline{\bu^{s}} )\nabla \bu^{s}
					=&(\overline{A_{31} }A_{31}+\overline{A_{32} }A_{41}) (\hat{r} \otimes \hat{r })+(\overline{A_{31} }A_{32}+\overline{A_{32} }A_{42} ) (\hat{r} \otimes \hat{\theta })\\
					&+(\overline{A_{41} }A_{31} +\overline{A_{42} }A_{41} )(\hat{\theta} \otimes \hat{r })+(\overline{A_{41} }A_{32} +\overline{A_{42} }A_{42})(\hat{\theta} \otimes \hat{\theta}).\notag
				\end{align}
				Therefore, substituting \eqref{eq:usxs}, \eqref{eq:graus} and \eqref{eq:grausgraus} into \eqref{eq:eueus}, one has 
				\begin{align}\label{eq:euzs}
					&E(\mathbf{u^{s}})\\
					=&-\lambda \int_{\mathcal{N}_{D,\gamma_{2}-1}^{+}}H_n(k_{p} r)k_{p} e^{in\theta }(a_{33}+a_{34} )(\overline{A_{31}}+ \overline{A_{42}})r\mathrm{d}r\mathrm{d}\theta\notag\\
					&+\mu  \int_{\mathcal{N}_{D,\gamma_{2}-1}^{+}}\left({\sf Tr}\left({\left(\nabla\mathbf{u^{s}}\right)^{H} }\nabla\mathbf{u}^{s}\right)+{\sf Tr}\left(\nabla\overline{\mathbf{u}^{s}}\nabla{\mathbf{u}^{s}}\right)\right)r\mathrm{d}r\mathrm{d}\theta\notag\\
					=&\left(\mathcal{G}_{1,n}\left(1-\frac{1}{\gamma_{2}^{2n+2}}\right)+\mathcal{G}_{2,n}\left(1-\frac{1}{\gamma_{2}^{2n}}\right)+\mathcal{G}_{3,n}\left(1-\frac{1}{\gamma_{2}^{2n-2}}\right)\right)\omega^{2n-2}\left(1+\mathcal{O}\left(\omega^2\right)\right)\notag\\
					=&\mathcal{G}\omega^{2n-2}\left(1+\mathcal{O}\left(\omega^2\right)\right),\notag
				\end{align}
				where
				\[\mathcal{G}_{1,n}+\mathcal{G}_{2,n}+\mathcal{G}_{3,n}=\mathcal{O}\left(\frac{  \kappa^2 2^{5-2 n}  \mu ^{2-n} \rho ^{n-1} }{ (n-1)!  (n-2)!}\right).\]
				
				When $n$ is sufficiently large, combining \eqref{eq:euzs} and \eqref{eq:graufsiu} yields
				\begin{align}\label{eq:eu2}
					\frac{E(\mathbf{u^{s}})}{\|\bu^i\|_{L^2(D)^2}^2}=&\frac{\mathcal{G}\omega ^{2n-2}(1+\mathcal{O}(\omega^2  ))}{\frac{\pi\kappa ^2 2^{3-2 n} \mu ^{1-n} \rho ^{n-1} }{n!(n-1)!} \omega^{2n-2}(1+\mathcal{O}(\omega^2  ))} =\mathcal{O}\left(n^2\right).
				\end{align}
				Under assumptions \eqref{eq:hicon}-\eqref{eq:de} and \eqref{eq:lambda=}, the  parameter $\delta$ satisfies $\delta\ll1$. Thus, for $n \geq \delta^{-1}$, we can obtain 
				\begin{displaymath}\notag  
					\frac{E(\mathbf{u^{s}})}{\|\bu^i\|_{L^2(D)^2}^2} = \mathcal{O}\left(n^2\right) \gg 1.  
				\end{displaymath} 
				Thus, we have demonstrated that the second formula in \eqref{eq:eug} holds.
				
				The proof is complete.
			\end{proof}
			\begin{rem}
				Combining \( \gamma_1 \in (0,1) \) with \( \delta \ll 1 \), a comparative analysis of \eqref{eq:grath1} and \eqref{eq:eug} demonstrates that
				\(
				\frac{1}{\delta^2} > \frac{1}{\delta},
				\)
				indicating that stress concentration represents a weaker condition compared to surface resonance. This conclusion shows that surface resonance inevitably leads to stress concentration.
			\end{rem}
			
			Theorem \ref{thm:eu} establishes that stress concentration occurs for the incident wave field $\mathbf{u}^i$ in \eqref{eq:ui} when the wave index satisfies $n \geq \delta^{-1}$. Proposition \ref{prop:4.3} demonstrates that these conditions can be strengthened to produce simultaneous boundary localization and stress concentration phenomena.
			\begin{prop}\label{prop:4.3}
				Under the same assumptions as Theorem~\ref{thm:eu}, when the incident wave index $n \geq \max\left(n_1, n_2, n_3, n_4, 1/\delta\right)$ for the wave field $\mathbf{u}^i$ specified in \eqref{eq:ui}, the  total field $\mathbf{u}$ and scattered field $\mathbf{u}^s$ satisfy the estimates \eqref{eq:def2.1} and \eqref{eq:eug}, respectively. This mathematically confirms the simultaneous occurrence of boundary localization and stress concentration.
			\end{prop}
			\begin{proof}
				Theorem \ref{thm:eu} establishes that stress concentration is induced when the incident wave field $\mathbf{u}^i$ in \eqref{eq:ui} has the index $n$ satisfying $n \geq \delta^{-1}$. The proof of Theorem \ref{thm:isl} reveals that the  condition $n \geq \max\left(n_1,n_2,n_3,n_4\right)$ specifically ensures boundary localization. By selecting the index $n$ satisfying the stronger condition $n \geq\max\left(n_1,n_2,n_3,n_4,\delta^{-1}\right)$, we simultaneously achieve both boundary localization and stress concentration phenomena.
			\end{proof}
			
			Theorem \ref{thm:eu} establishes that in the sub-wavelength regime with prescribed high contrast material parameters, proper selection of the incident wave index $n$ in \eqref{eq:ui} induces quasi-Minnaert resonance. Remarkably, these resonant effects conversely provide critical design criteria for optimizing high contrast material configurations, analogous to the framework developed in Proposition \ref{prop:4.2}.
			\begin{prop}
				Under the assumptions of Theorem \ref{thm:eu} with boundary localization level $\varepsilon \ll 1$, the material parameter $\delta$ can be optimally selected through condition \eqref{eq:327} to engineer structures with tailored wave properties. When the incident wave index $n$ in \eqref{eq:ui} satisfies $n \geq \max\left(n_2,n_4,\delta^{-1}\right)$, the designed structure simultaneously manifests both boundary localization and stress concentration phenomena.
			\end{prop}
			\begin{proof}
				Following the methodology established in  Proposition \ref{cor:isl} and Proposition \ref{prop:4.2}, the current proof can be completed through analogous arguments. For conciseness, we omit the detailed derivation here.
			\end{proof}
		
\section*{Concluding Remarks} 
\label{sec:conclusions}

In this paper, we have provided a detailed and systematic analysis of quasi-Minnaert resonance. We have demonstrated that the boundary localization and surface resonance depend critically on both the appropriate selection of incident waves and the high contrast in physical parameters between the rigid elastic material and the soft elastic medium. The stress concentration phenomena in both the internal total field and the scattered field  are rigorously established.  By integrating layer potential techniques, refined asymptotic analysis, and carefully constructed incident waveforms, we provide a rigorous framework for establishing quasi-Minnaert resonance in both internal and scattered wave fields.


\bigskip

\noindent\textbf{Acknowledgment.}
The work of H. Diao is supported by the National Natural Science Foundation of China  (No. 12371422) and the Fundamental Research Funds for the Central Universities, JLU.

\medskip




\end{document}